\documentclass[12pt]{amsart}

\usepackage{amsmath,amsthm}
\usepackage{amssymb}
\usepackage{hyperref}
\usepackage{color}
\usepackage{graphicx}

\newtheorem{theorem}{Theorem}[section]
\newtheorem{corollary}[theorem]{Corollary}
\newtheorem{lemma}[theorem]{Lemma}
\newtheorem{proposition}[theorem]{Proposition}

\theoremstyle{definition}
\newtheorem{definition}[theorem]{Definition}
\newtheorem{remark}[theorem]{Remark}

\def\tangle#1#2{\left[\frac{#1}{#2}\right]}

\title[Embeddings of 2-string tangles]{On embeddings of $2$-string tangles into the unknot, the unlink and split links}
\author[J. M. Nogueira and A. Salgueiro]{Jo\~ao M. Nogueira and Ant\'onio Salgueiro}
\thanks{This work was partially supported by the Centre for Mathematics of the University of Coimbra - UIDB/00324/2020, funded by the Portuguese Government through FCT/MCTES}

\begin{document}

\begin{abstract}
In this paper we study further when tangles embed into the unknot, the unlink or a split link. In particular, we study obstructions to these properties through geometric characterizations, tangle sums and colorings. As an application we determine when each prime 2-string tangle with up to seven crossings embeds into the unknot, the unlink or a split link. 
\end{abstract}

\maketitle

\section{Introduction}

A {\em tangle} $\mathcal{T}$ is a pair $(B, \sigma)$ formed by a ball $B$ and a collection of properly embedded disjoint arcs $\sigma$ in $B$. If $\sigma$ has $n$ components we say that $\mathcal{T}$ is a $n$-string tangle.
Two tangles $\mathcal{T}$, $\mathcal{T}'$ are {\em equivalent}, denoted by $\mathcal{T}\approx\mathcal{T}'$, if there is an isotopy of $B$ sending $\mathcal{T}$ to $\mathcal{T}'$, and {\em strongly equivalent}, denoted by $\mathcal{T}=\mathcal{T}'$, if this isotopy fixes $\partial B$. For example, all rational $2$-string tangles are equivalent, but they are not strongly equivalent.

Let $K$ be a link in $S^3$, and $B$ a ball in $S^3$ with exterior $B'$. If $\mathcal{T}=(B,B\cap K)$ and $\mathcal{T}'=(B',B'\cap K)$ are tangles, then we say that $\mathcal{T}\cup\mathcal{T}'$ is a {\em tangle decomposition} of $K$ and that $K$ is a {\em closure} of $\mathcal{T}$ (and of $\mathcal{T}'$). In case there is a tangle decomposition of $K$ with $\mathcal{T}$ one of the tangle components, we also say that $\mathcal{T}$ embeds into the pair $(S^3, K)$, or, for abbreviation, that it embeds into  $K$. Similarly, let $\mathcal{T}_1=(B_1,\sigma_1)$ and $\mathcal{T}_2=(B_2,\sigma_2)$ be tangles such that $B_1\cap B_2$ is a disk and $\mathcal{T}=(B_1\cup B_2,\sigma_1\cup \sigma_2)$ is a tangle. Then we say that $\mathcal{T}_1\cup \mathcal{T}_2$ is a {\em tangle decomposition} of  $\mathcal{T}$.\\

A fundamental question in knot theory is determining whether a knot or link is actually the unknot (resp., an unlink or a split link). In this paper, we continue the study of when a tangle embeds or not into the unknot, the unlink or a split link. If a tangle $\mathcal{T}$ embeds into the unknot, the unlink, or a split link, we say that $\mathcal{T}$ is {\em unknottable},  {\em unlinkable} or {\em splittable}. We refer to a tangle $\mathcal{U}$ such that $\mathcal{T}\cup\, \mathcal{U}$ is a tangle decomposition of the unknot, the unlink or of a split link, respectively, as an {\em unknotting}, {\em unlinking} or {\em splitting} {\em closure tangle} of $\mathcal{T}$.

A {\em projection} of a tangle $\mathcal{T}$ is the image $p(\mathcal{T})$ of the tangle by an orthogonal projection $p$ to a plane such that the preimage of each point of $p(\mathcal{T})$ has at most two points, and there are a finite number of double points, which are called the {\em crossings} of the projection. A projection always exists in the piecewise linear category. 
If these crossings are decorated with broken lines to show the overcrosses and undercrosses, then we get a {\em diagram} of $\mathcal{T}$. 
A tangle diagram $\mathcal{D}$ is called {\em unknottable}, {\em unlinkable} or {\em splittable}  if there is a diagram of the unknot, the unlink or a split link, respectively, that  contains $\mathcal{D}$. (See Figure \ref{figure:51ab} for an example of an unknottable tangle diagram.)

\begin{figure}[ht]
	\centering
	\includegraphics[width=0.5\textwidth]{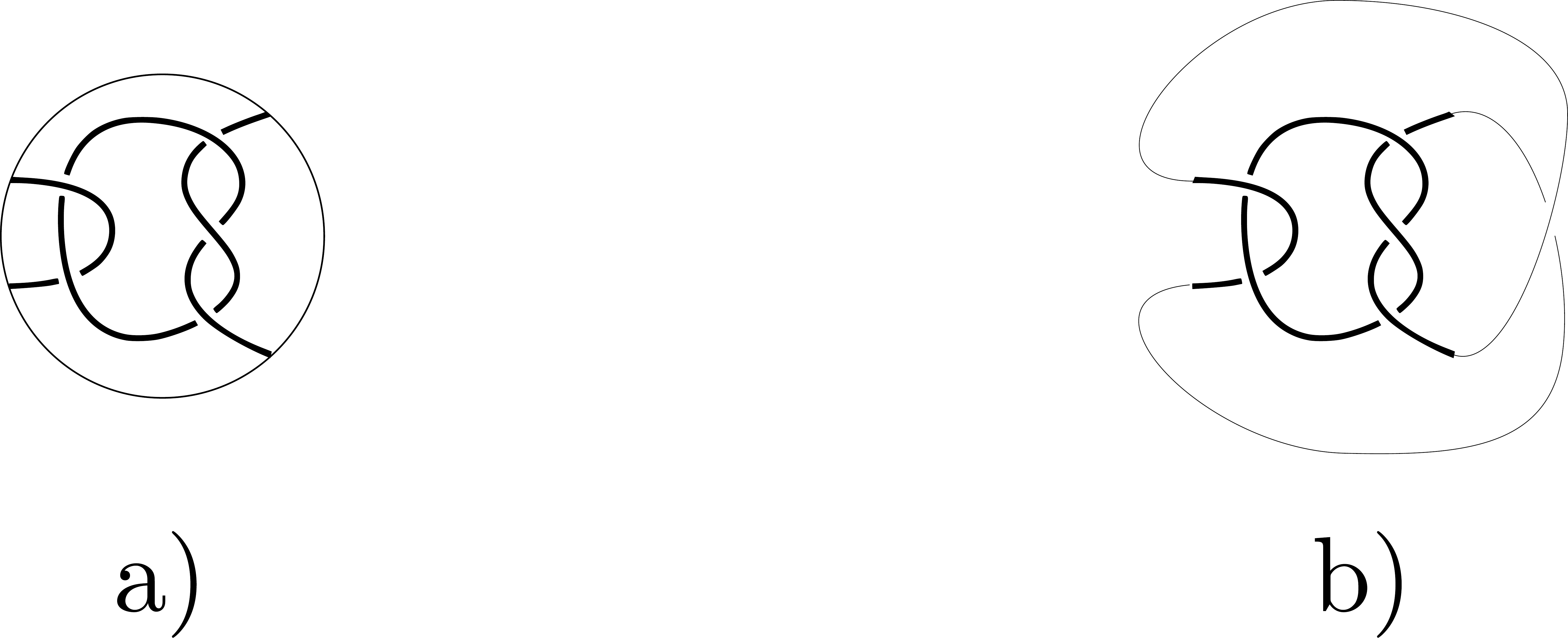} 
	\caption{The diagram a) is unknottable since it is contained on the diagram b) of the unknot.}
	\label{figure:51ab}
\end{figure}

The study of embeddings of tangles into the unknot, the unlink or split links has been considered before in several papers. In \cite{K}, Krebes shows that the greatest common divisor of the determinant of the numerator and denominator closure of a 2-string tangle divides the determinant of any knot or link obtained by closing the tangle, and is able to show several tangles not to be unknottable. This work uses a combinatorial interpretation of the determinant in terms of link diagrams and Kauffman brackets. In \cite{SW-18}, Silver and Williams gave a new shorter proof of this result. In \cite{KSW-99}, this approach is generalized to $n$-string tangles with certain bracket-derived invariants of link diagrams.  In \cite{SW-99}, Silver and Williams extend Krebes' result using Fox colorings. More recently, Kauffman and Lopes \cite{KL-20} use colorings of knot diagrams with involutory quandles to study how to obtain tangles that are not unknottable. In \cite{R-00}, Ruberman extends the work of Krebes following a homological interpretation of the determinant of a link as the order of the first homology of the 2-fold branched cover of $S^3$ over the link; and also introduces an obstruction for a 2-string tangle to be unlinkable, through an application of work of Cochran and Ruberman \cite{CR-89} on tangle invariants from higher order linking numbers. In \cite{PSW-05}, Przytycki, Silver and Williams extend the work of Krebes and of Ruberman to $n$-string tangles and use the Jones polynomial instead of just the determinant. With this work, they are able to show an example of a 2-string tangle from  \cite{K} that is not unknottable and unlinkable and which is not possible to prove using the determinant result of \cite{K}. In the present paper, we are also able to show that this tangle is not unknottable and unlinkable with a different approach.\\

For 2-string tangles in particular, more can be said from the literature. A 2-string tangle $(B, \sigma)$ is called {\em essential} if its strings cannot be separated by a disk properly embedded in $B$ and {\em inessential} otherwise. As observed in section \ref{section: algebraic properties}, a 2-string tangle is unknottable (resp., unlinkable) if and only if the unknotting (resp., unlinking) closure tangle is a rational tangle, and it is splittable if and only if it has a rational tangle as a splitting closure tangle. Under these circumstances, with respect to strong equivalence, the unknotting closure tangle  of an unknottable 2-string essential tangle is unique \cite{BS-86,BS-88}, and the unlinking (resp. rational splitting) closure tangle of an unlinkable (resp., splittable) 2-string tangle is unique \cite{EM-88}. Furthermore, in case the 2-string tangle is essential, it cannot be unknottable and splittable (or unlinkable) simultaneously \cite{S-85}. New proofs of these results were also obtained by Taylor in \cite{T-08}. For the case of a rational tangle it is possible to determine exactly which rational tangles are its unknotting closure tangles, as seen in the work of Ernst and Sumners \cite{ES-90} and of Kauffman and Lambropoulou in \cite{KL-11}.\\

This paper is organized as follows. In Section \ref{section: properties}, we observe some  fundamental basic properties on embeddings of tangles into the unknot, the unlink and a split link. We also observe that Conjecture 3.1 of \cite{KL-20} by Kauffman and Lopes is false. In Section \ref{section:geometric}, we give a geometric characterization for a 2-string tangle to be unknottable, unlinkable or splittable. In Section \ref{section: algebraic properties}, we study the behavior of unknottability, unlinkability and unsplittability  under the sum of tangles. In particular, we determine when a Montesinos tangle is unknottable, unlinkable or splittable. In Section \ref{section:colorings}, we study further when the coloring invariants are an obstruction for a tangle being unknottable. We also apply these invariants on the study of unlinkable tangles for the first time. In particular, we determine the unlinking closure tangle candidate for any 2-string tangle. In Section \ref{section:tables}, we determine which tangles of the table classifying all 2-string tangles up to 7 crossings (table 1 in \cite{KSS}) are unknottable, unlinkable or unsplittable. We refer to this table throughout the paper.

\section{Basic properties}\label{section: properties}

In this section, we observe basic properties of unknottable, unlinkable and splittable tangles.
\begin{theorem}\label{theorem:equivalence} Let $\mathcal T$ be a tangle. The following are equivalent:
	\begin{enumerate}
		\item for every tangle $\mathcal T'$ (strongly) equivalent to $\mathcal T$, every diagram $\mathcal D'$ of $\mathcal T'$ is unknottable;
		\item there is a tangle $\mathcal T'$ (strongly) equivalent to $\mathcal T$ that has a unknottable diagram;		
		\item $\mathcal T$ is unknottable.
	\end{enumerate}
	A similar result holds for unlinkable/splittable tangles.
\end{theorem}

\begin{proof}
	The implication (a)$\Rightarrow$(b) is immediate, by considering any diagram of $\mathcal T$. 
	
	To prove (b)$\Rightarrow$(c), consider an unknottable diagram $\mathcal D'$ of $\mathcal T'$. Then there is a diagram $\mathcal D$ of the unknot $K$ that contains $\mathcal D'$. Consider an isotopy $\lambda$ of $B$ sending $\mathcal T'$ to $\mathcal T$. Extend $\lambda$ to $S^3$ by defining it on the exterior $B'$ of $B$ as the conjugate of $\lambda$ by a symmetry interchanging $B$ and $B'$. Then $\lambda(K)$ is the unknot with a decomposition $\mathcal{T}\cup\mathcal{U}$. Therefore, $\mathcal T$ is unknottable.
	
	Finally, to prove (c)$\Rightarrow$(a), suppose that $\mathcal T$ is unknottable. Then there is a tangle $\mathcal{U}$ such that $\mathcal{T}\cup\mathcal{U}$ is a tangle decomposition of the unknot. Consider an isotopy $\lambda$ of $B$  sending $\mathcal T$ to $\mathcal T'$ (and its extension to $S^3$) and a diagram $\mathcal D'$ of $\mathcal T'$. Then $\lambda(\mathcal{T})\cup\lambda(\mathcal{U})$ is a tangle decomposition of the unknot that has a diagram containing $\mathcal D'$. Therefore $\mathcal D'$ is unknottable.
\end{proof}

We say that the tangle $\mathcal{T}=\mathcal{T}_1\cup \mathcal{T}_2$ is the {\em union} of the tangles $\mathcal{T}_1=(B_1,\sigma_1)$ and $\mathcal{T}_2=(B_2,\sigma_2)$ if $B_1\cap B_2$ is a disk disjoint from  $\sigma_1\cup\sigma_2$. If a tangle $\mathcal{T}_1=(B_1,\sigma_1)$ is embedded in a tangle $\mathcal{T}=(B,\sigma)$, we say that $\mathcal{T}_1$ is a {\em subtangle} of $\mathcal{T}$ (See Figure \ref{figure:decompostion_and_union}.)

\begin{figure}[ht]
	\centering
	\includegraphics[scale=.1]{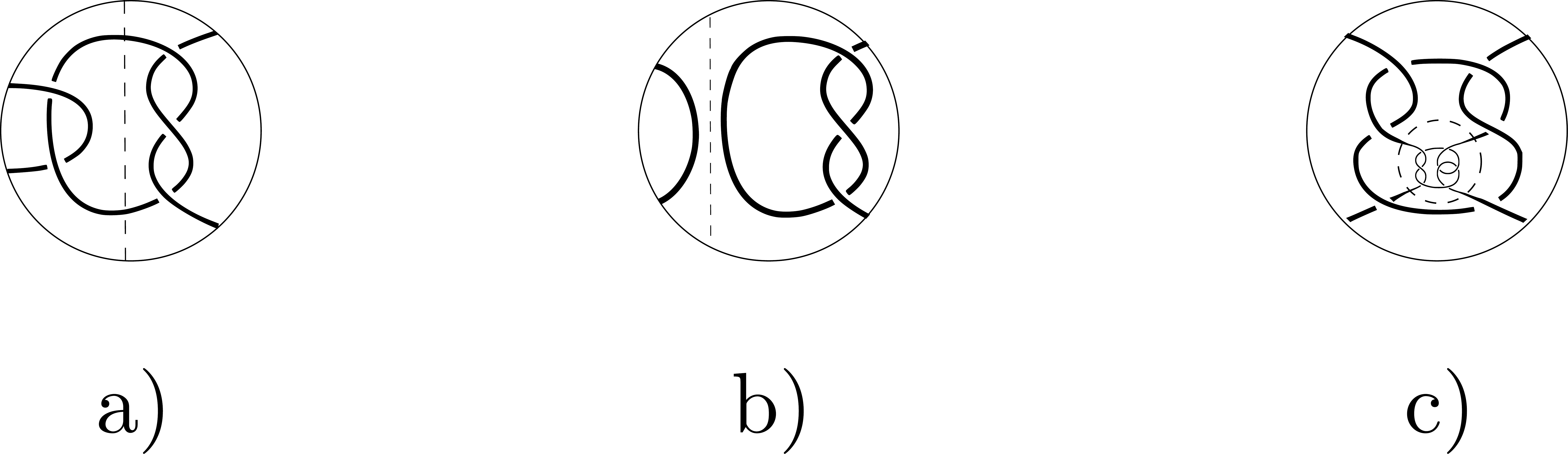} 
	\caption{a) A decomposition of tangles; b) an union of tangles; c) a subtangle.}
	\label{figure:decompostion_and_union}
\end{figure}

\begin{theorem}\label{theorem:subtangle}
	Let $\mathcal{T}_1$ be a subtangle of $\mathcal{T}$. \begin{enumerate}
		\item If $\mathcal{T}$ is unknottable, then  $\mathcal{T}_1$ is unknottable.
		\item If $\mathcal{T}$ is unlinkable, then  $\mathcal{T}_1$ is unlinkable if it intersects more than one string of $\mathcal{T}$; otherwise $\mathcal{T}_1$ is unknottable.
		\item If $\mathcal{T}$ is splittable and $\mathcal{T}_1$ intersects two strings of $\mathcal{T}$ separated by the splitting decomposition, then  $\mathcal{T}_1$ is splittable.
	\end{enumerate}
\end{theorem}

\begin{proof}
	If $\mathcal{T}\cup\mathcal{T'}$ is a tangle decomposition of a link $K$, then  $\mathcal{T}_1\cup \left(\mathcal{T}'\cup (\mathcal{T}-\mathcal{T}_1)\vphantom{2^{2^2}}\right)$ is also a tangle decomposition of $K$.
	
	If $K$ is an unknot, then $\mathcal{T}_1$ is unknottable; if $K$ is an unlink, then $\mathcal{T}_1$ is unlinkable if it intersects more than one string of $\mathcal{T}$ and unknottable otherwise; if $K$ is a split link, then $\mathcal{T}_1$ is splittable if it intersects  two strings of $\mathcal{T}$ separated by the splitting decomposition.
\end{proof}

As an immediate consequence of this theorem, we have the following corollary.

\begin{corollary}\label{corollary:subtangle}
	Let  $\mathcal{T}_1 \cup\mathcal{T}_2$ be a tangle decomposition of  $\mathcal{T}$.
	\begin{enumerate}
		\item If $\mathcal{T}$ is unknottable, then  $\mathcal{T}_1$ and  $\mathcal{T}_2$ are unknottable.
		\item If $\mathcal{T}$ is unlinkable, then $\mathcal{T}_1$ and $\mathcal{T}_2$ are unlinkable or unknottable. If the tangle decomposition is not an union, then at least one of 	$\mathcal{T}_1$ or $\mathcal{T}_2$ is unlinkable.
		\item If $\mathcal{T}$ is splittable and the tangle decomposition is not an union, then at least one of 	$\mathcal{T}_1$ or $\mathcal{T}_2$ is splittable.
	\end{enumerate}
\end{corollary}

For instance, since $7_{16}\approx\mathcal{T}*[-2]$, where $\mathcal{T}\approx 6_2$, as illustrated in Figure  \ref{figure:7_16_as_product}, and the tangle $6_2$ is not unknottable nor unlinkable (see Section \ref{section:tables}), then $7_{16}$ is also not unknottable nor unlinkable (as both strings are unknotted, from Proposition \ref{proposition:rational}, it is also not splittable).

\begin{figure}[ht]
	\centering
	\includegraphics[scale=.1]{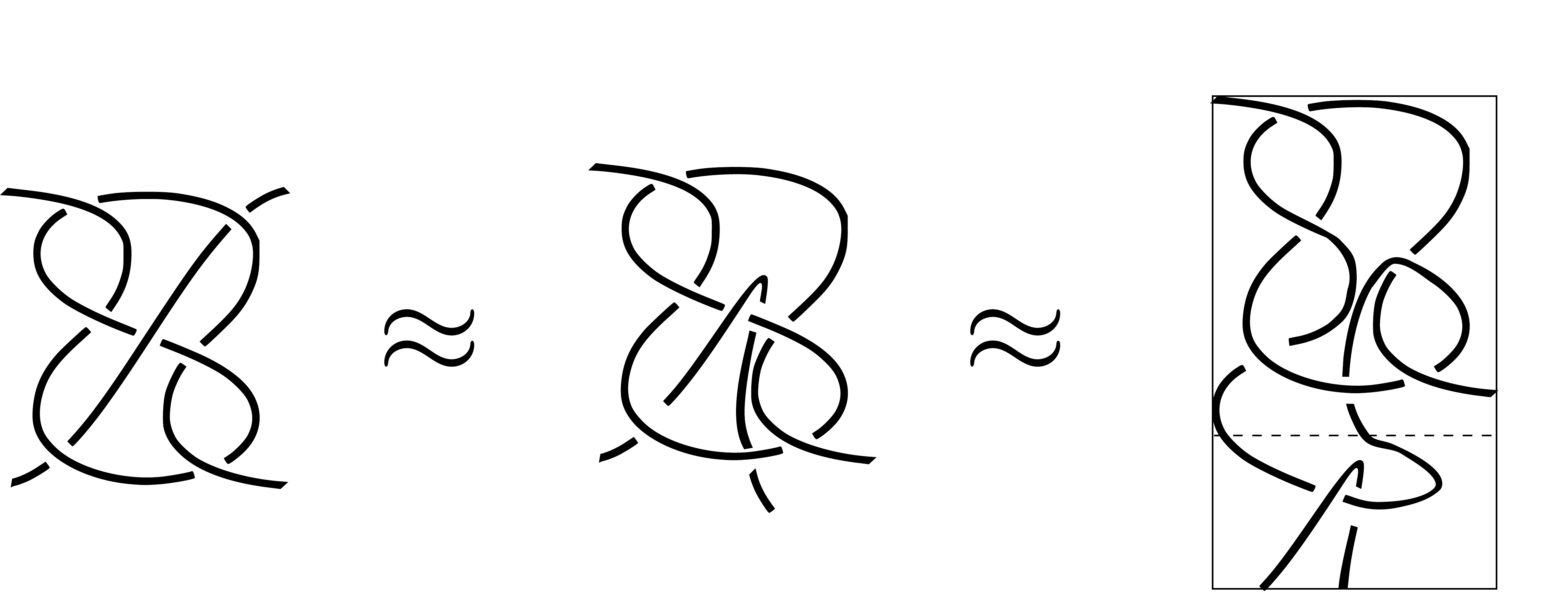}
	\caption{The equivalence of $7_{16}$ to a product of a tangle equivalent to $6_2$ and $[-2]$.}
	\label{figure:7_16_as_product}
\end{figure}

The converse of this corollary is not true. For instance,  $6_2$ is not unknottable, unlinkable or splittable, as verified in Section \ref{section:tables}, but it has a decomposition into two trivial 2-string tangles.
However, in the case where the decomposition of tangles is an union, the  converse is also true, as stated in the following theorem.

\begin{theorem}\label{lemma:side_by_side}
	Let $\mathcal{T}$ be the union of the tangles  $\mathcal{T}_1$ and $\mathcal{T}_2$. Then
	\begin{enumerate}
		\item  $\mathcal{T}$ is unknottable if and only if $\mathcal{T}_1$ and $\mathcal{T}_2$ are unknottable.
		\item  $\mathcal{T}$ is unlinkable if and only if $\mathcal{T}_1$ and $\mathcal{T}_2$ are unlinkable or unknottable.
		\item  $\mathcal{T}$ is splittable.
	\end{enumerate} 
\end{theorem}

\begin{proof}
	(a) Suppose that $\mathcal{T}$ is unknottable. By Corollary \ref{corollary:subtangle}a), $\mathcal{T}_1$  and $\mathcal{T}_2$ are unknottable. 
	Now suppose that $\mathcal{T}_1$ and $\mathcal{T}_2$ are unknottable. Then there exist
	tangles $\mathcal{T}'_1,\mathcal{T}'_2$ such that $\mathcal{T}_1\cup\mathcal{T}'_1$ and $\mathcal{T}_2\cup\mathcal{T}'_2$ are tangle decompositions of the unknot. By adding a single crossing between $\mathcal{T}'_1$ and $\mathcal{T}'_2$, as in Figure \ref{figure:unknottable}, we obtain an unknot $K$ containing $\mathcal{T}$. Therefore $\mathcal{T}$ is unknottable.\\
	
	\begin{figure}[ht]
		\centering
		\includegraphics[scale=.075]{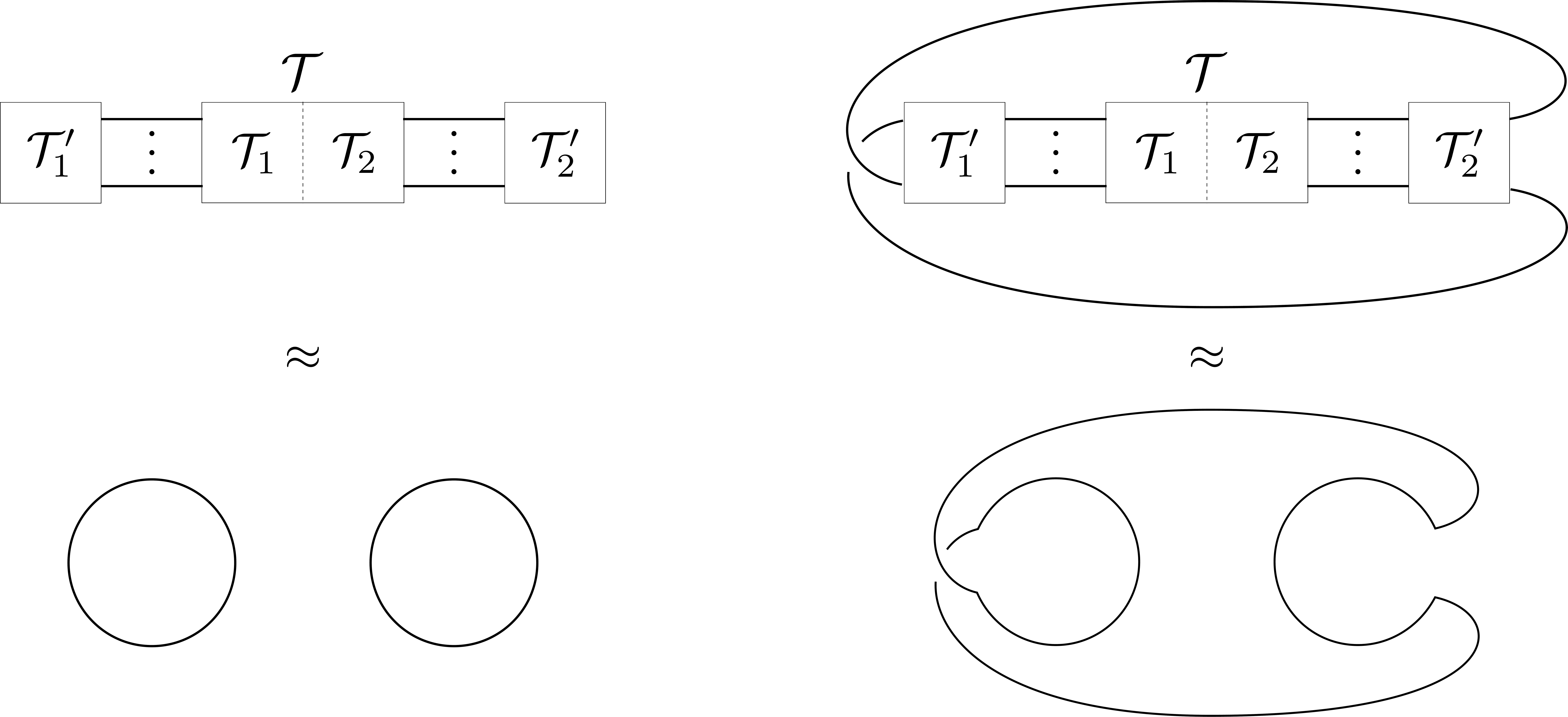} 
		\caption{The union of unknottable tangles is unknottable.}
		\label{figure:unknottable}
	\end{figure}
	
	(b) Suppose that $\mathcal{T}$ is unlinkable. By Corollary \ref{corollary:subtangle} (b), $\mathcal{T}_1$  and $\mathcal{T}_2$ are unlinkable or unknottable.
	Now suppose that $\mathcal{T}_1$ and $\mathcal{T}_2$ are unlinkable or unknottable. Then there exist tangles $\mathcal{T}'_1,\mathcal{T}'_2$ such that $\mathcal{T}_1\cup\mathcal{T}'_1$ and $\mathcal{T}_2\cup\mathcal{T}'_2$ are tangle decompositions of the unknot or an unlink. Hence, the union of the tangles $\mathcal{T}'_1$, $\mathcal{T}'_2$ and $\mathcal{T}$ is an unlink. Therefore $\mathcal{T}$ is unlinkable.\\
	
	(c) From the decomposition of $\mathcal{T}$ defined by $\mathcal{T}_1$ and $\mathcal{T}_2$ we obtain a split link by closing $\mathcal{T}$ appropriately on each side of the decomposition.
\end{proof}

\begin{corollary}
	Any $n$-string tangle equivalent to the trivial $n$-string tangle is unknottable and unlinkable.
\end{corollary}

By this Corollary, any rational tangle is unknottable. (See also \cite{ES-90, KL-11}.) For 2-string tangles the converse of this corollary is true, as a 2-string tangle that is unknottable and splittable is trivial \cite{EM-88}.



But in general, the  converse of this corollary is not true. In fact, consider the 3-string tangle $\mathcal T$ defined by adding a parallel string to one of the strings of the 2-string tangle $6_3$. Then $\mathcal T$ is unlinkable (it embeds into the three components unlink) and unknottable, but it is not equivalent to a trivial tangle. 

We also observe that Conjecture 3.1 of \cite{KL-20} is false. In fact, the tangle $\mathcal T$ of Figure \ref{figure:51ab}a) is essential \cite{NS}, but it is unknottable, as can be seen in Figure \ref{figure:51ab}b), which shows a diagram of a trivial knot containing a diagram of $\mathcal T$. Moreover,  $\mathcal T$ is irreducible in the sense of \cite{KL-20}, because its numerator closure is the nontrivial knot $5_1$, as illustrated in Figure \ref{figure:51c}, and $\mathcal T$ has minimum crossing number \cite{NS}. 

\begin{figure}[ht]\label{figure:conjecture}
	\centering
	\includegraphics[scale=.1]{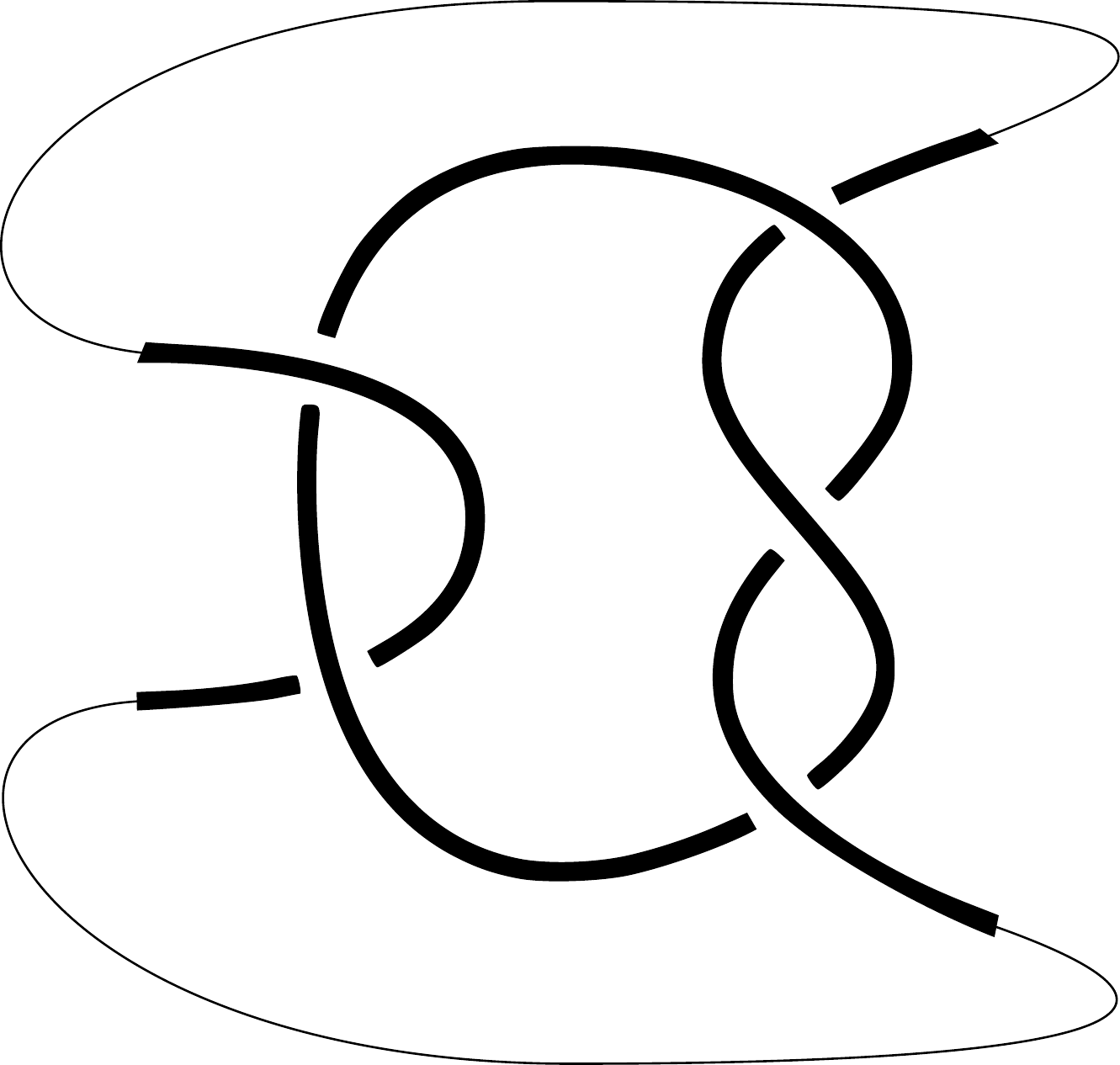} 
	\caption{The numerator closure of $\mathcal{T}$ .}
	\label{figure:51c}
\end{figure}

\section{Geometric characterization for 2-string tangles} \label{section:geometric}

In this section we characterize geometrically unknottable, unlinkable and splittable 2-string tangles. We will use the expression {\em (punctured) disk} to refer to a disk minus the interior of a set of $n$ disks, where $n$ is a non-negative integer (i.e., to either a disk or a punctured disk).

\begin{definition}
	Let $\mathcal{T}=(B,s_1\cup s_2)$ be a 2-string tangle.\\ 
	We say that $s_1$ and $s_2$ are {\em quasi-parallel} if there is a (punctured) disk $P$ embedded in $B$ such that $s_1$ and $s_2$ are in the same component of $\partial P$ and $\partial P-(s_1\cup s_2)\subset\partial B$.\\
	We say that $s_1$ and $s_2$ are {\em quasi-trivial} if there are two disjoint (punctured) disks $P_1$ and $P_2$ embedded in $B$ co-bounded, respectively, by $s_1$ and $s_2$.\\
	We say that $\mathcal{T}$ is {\em quasi-inessential} if there is a (punctured) disk $D$ properly embedded in $B$, disjoint from $s_1$ and $s_2$, such that each boundary component of $D$ separates the ends of $s_1$ from the ends of $s_2$. (See Figure \ref{figure:quasiparallel}.)
\end{definition}

\begin{figure}[ht]
	\centering
	\includegraphics[scale=.075]{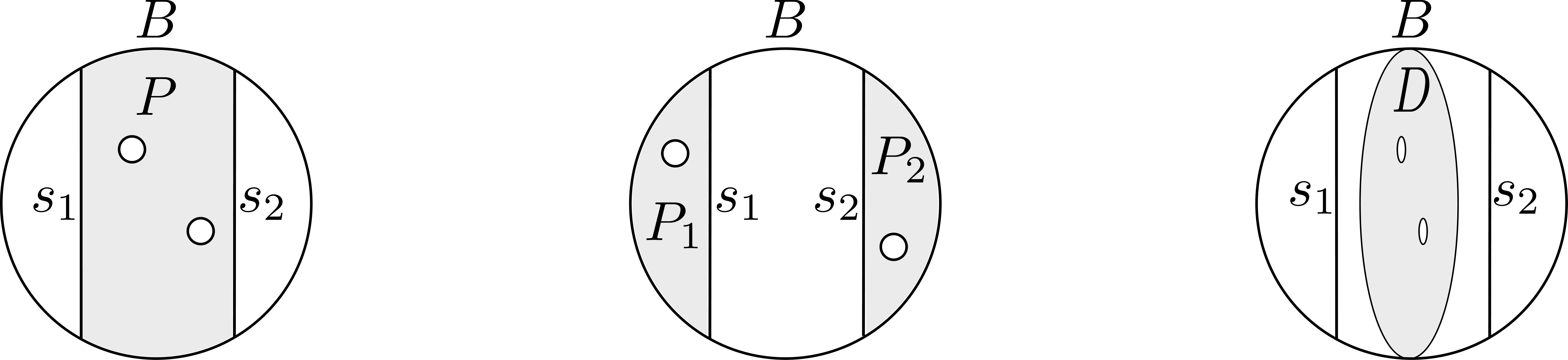} 
	\caption{Quasi-parallel, quasi-trivial and quasi-inessential strings.}
	\label{figure:quasiparallel}
\end{figure}

We remark that a rational tangle is quasi-inessential and its strings are quasi-parallel and quasi-trivial, with the disks in each case having no punctures. We can visualize these disks by starting with two parallel strings as in Figure \ref{figure:quasiparallel} and rotating the ends of the strings until obtaining the desired rational tangle.
	
For non rational tangles, punctures are required as the disks are pushed across $\partial B$. On Figure \ref{figure:puncture}, we can visualize such a puncture, which is defined by the intersection of the disk bounded by the depicted unknot with $\partial B$.

\begin{figure}[ht]
		\centering
		\includegraphics[scale=.15]{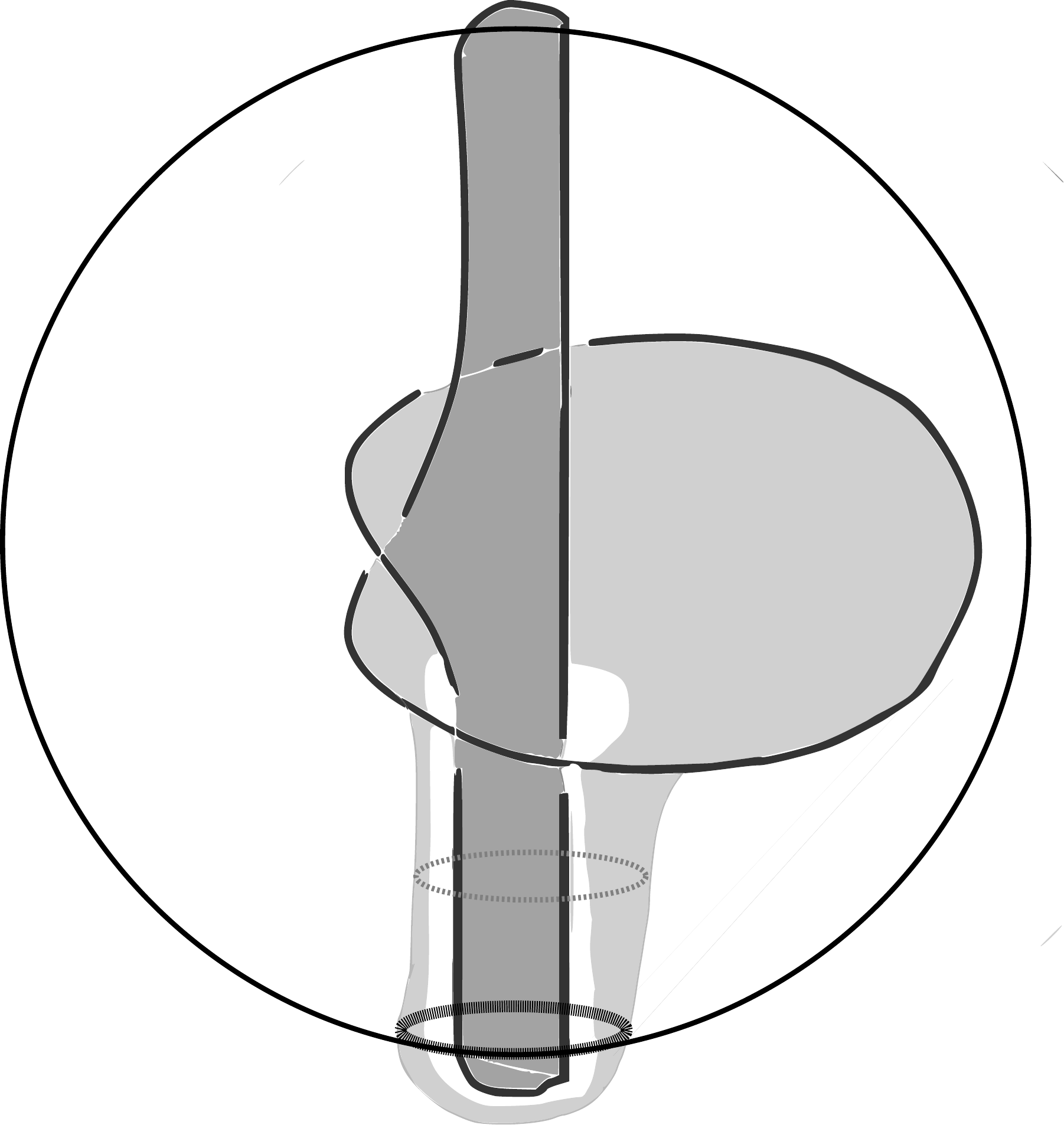} 
		\caption{A punctured disk in $B$.}
		\label{figure:puncture}
\end{figure}

\begin{lemma}\label{lemma:quasi-parallel}
	Let $\mathcal{T}=(B,s_1\cup s_2)$ be a 2-string tangle, with $s_1$ and $s_2$ quasi-parallel. If $\mathcal{T}$ is inessential, then it is trivial.
\end{lemma}

\begin{proof}
	Since $s_1$ and $s_2$ are quasi-parallel, there exists a (punctured) disk $P$ embedded in $B$ such that $s_1$ and $s_2$ are in the same component $b$ of $\partial P$ and $\partial P-(s_1\cup s_2)\subset\partial B$. Assume that $P$ has the minimal possible number of punctures.
	
	Since $\mathcal{T}$ is inessential, there exists a disk $D$ properly embedded in $B$ separating the strings $s_1$ and $s_2$. Assume that the number of components of $D\cap P$, denoted by $|D\cap P|$, is minimal, among all possible disks $D$. Let $a$ and $a'$ be the arcs defined by $b-(s_1\cup s_2)$ and $\alpha$ be an outermost arc of $D\cap P$ in $D$.
	
	Suppose that $\alpha$ is non-separating in $P$. Then, by compressing $P$ along $\alpha$ and the corresponding outermost disk in $D$, we reduce the number of punctures of $P$, contradicting its minimality. Otherwise, suppose that $\alpha$ is separating in $P$. If the ends of $\alpha$ are in the same component of $\partial P - (s_1\cup s_2)$,  by compressing $P$ along $\alpha$ and the corresponding outermost disk in $D$, we reduce the number of punctures of $P$ or reduce $|D\cap P|$, contradicting their minimality. Hence, $\alpha$ is separating and has ends in distinct components of $\partial P - (s_1\cup s_2)$. As $\alpha$ is separating, it has to have ends in the same component of $\partial P$. Hence, $\alpha$ has ends in $b$, more precisely, one end in $a$ and the other end in $a'$. Then, $\alpha$ cuts $P$ into two disjoint embedded possibly punctured disks, one co-bounded by $s_1$ and the other by $s_2$. Therefore, $s_1$ and $s_2$ are unknotted, which, for being inessential, implies that $\mathcal{T}$ is trivial. \end{proof}

\begin{theorem} \label{theorem:main}A 2-string tangle is unknottable if and only if its strings are quasi-parallel.
\end{theorem}

\begin{proof} Let $\mathcal{T}$ be a 2-string tangle.
	Suppose that $\mathcal{T}$ is inessential. Then $\mathcal{T}$ is the union of two $1$-string tangles $\mathcal{T}_1$ and $\mathcal{T}_2$. By Theorem \ref{lemma:side_by_side}, $\mathcal{T}$ is unknottable if and only if the tangles $\mathcal{T}_1$ and $\mathcal{T}_2$ are unknottable, that is, both strings are unknotted. By Lemma \ref{lemma:quasi-parallel}, this is equivalent to the strings being quasi-parallel.
	
	Now consider the case when $\mathcal{T}=(B,s_1\cup s_2)$  is essential. 
	Suppose first that $\mathcal{T}$ is unknottable and let $\mathcal{U}=(B',u_1\cup u_2)$ be its unknotting closure tangle, with $\mathcal{T}\cup\mathcal{U}$ a tangle decomposition of the unknot $K$. Denote the sphere $B\cap B'$ by $S$. Since $K$ is trivial, it bounds an embedded disk $D$ in $S^3$. Consider $D$ such that the number of components of $D\cap S$, denoted by $|D\cap S|$, is minimal. 
	
	Let $\alpha$ be one of those components. If $\alpha$ is a closed curve, then it is essential in $S-(s_1\cup s_2)$, otherwise $\alpha$ bounds a disk in $S-(s_1\cup s_2)$ and we can reduce $|D\cap S|$, contradicting its minimality. As  $\mathcal{T}$ is essential, the disk that $\alpha$ bounds in $D$ is in $B'$. Therefore, by removing all these disks from $D$, we obtain a (punctured) disk $D'$ embedded in $B$.
	
	Now let $\alpha$ be an arc component of $D\cap S$. Since the only intersections of $\partial D$ and $S$ are the endpoints of $s_1$ and $s_2$, the arc $\alpha$ connects these endpoints and there are only two such arcs $\alpha_1$ and $\alpha_2$.
	
	Suppose first that $\alpha_i$ connects the endpoints of $s_i$, for $i=1,2$. Let $D_1$, $D_2$ be disjoint disks cut from $D$ by $\alpha_1$, $\alpha_2$.
	If $D_i\cap S$ contains a closed curve $\alpha$, as in Figure \ref{figure:disk}, then, as observed before, $\alpha$ doesn't bound a disk in $S-(s_1\cup s_2)$, it bounds a disk in $B'$. As $D_1$ is disjoint from $D_2$, the curve $\alpha$ separates $\partial s_1$ from $\partial s_2$ in $S$, and the disk $\alpha$ bounds in $B'$ separates $u_1$ and $u_2$. Then, without loss of generality, $\partial u_i$ is attached to $\partial s_i$, but this contradicts $K$ being a knot. Therefore, $D_i$ is disjoint from $S$, hence $s_i$ is trivial in $\mathcal{T}$, which contradicts $\mathcal{T}$ being essential.
	
	\begin{figure}[ht]
		\centering
		\includegraphics[scale=.1]{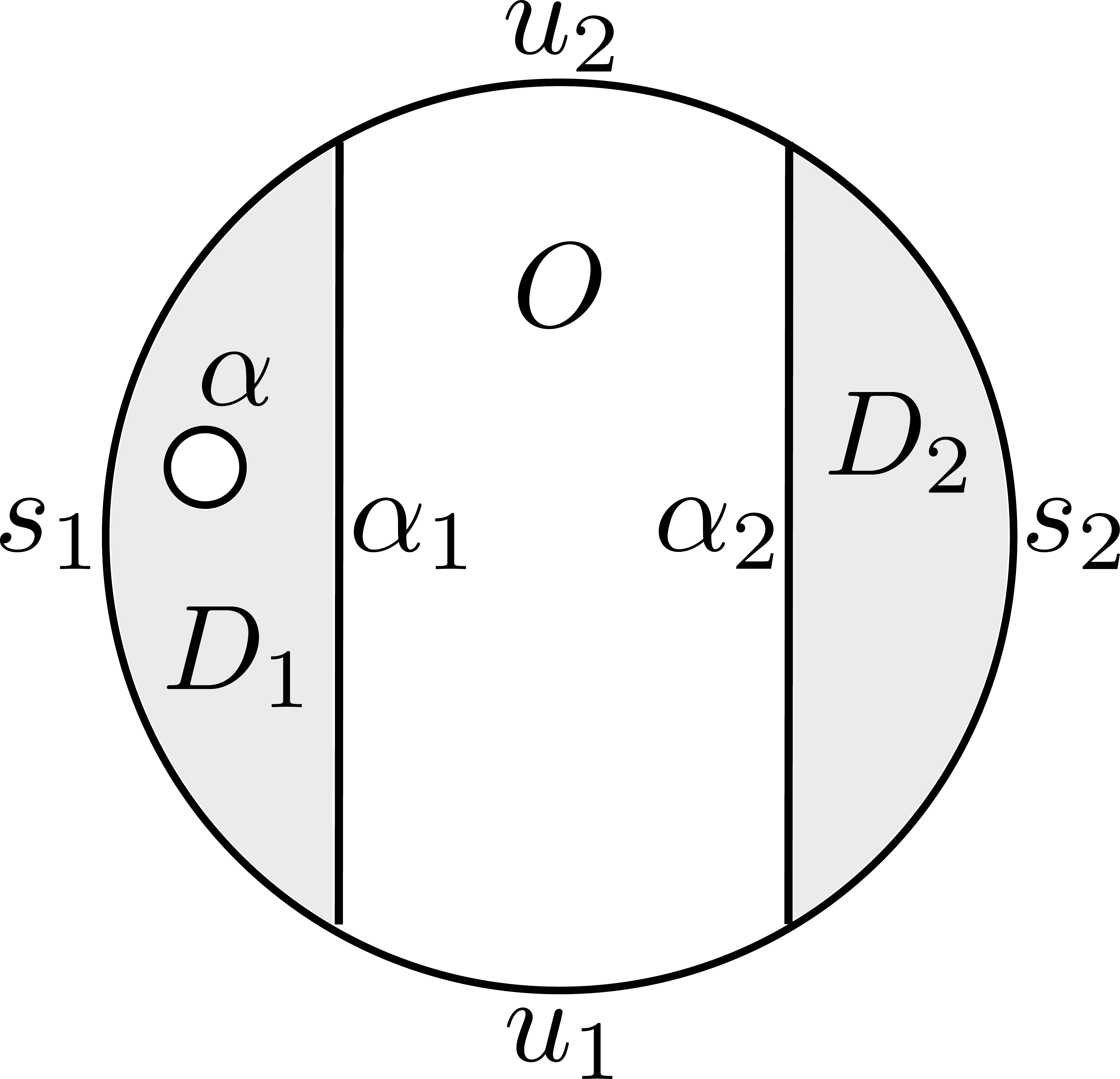} 
		\caption{If $\alpha_i$ connects the ends of $s_i$.}
		\label{figure:disk}
	\end{figure}
	
	Suppose now that $\alpha_1$ and $\alpha_2$ both connect an endpoint of $s_1$ and an endpoint of $s_2$. Let 
	$P$ be the (punctured) disk cut from $D'$ by $\alpha_1$ and $\alpha_2$, as in Figure \ref{figure:disk2}. Then $P$ is a disk embedded in $B$ such that $s_1$ and $s_2$ are in the same component of $\partial P$ and $\partial P-(s_1\cup s_2)\subset\partial B$. Therefore, $s_1$ and $s_2$ are quasi-parallel.
	
	\begin{figure}[ht]
		\centering
		\includegraphics[scale=.1]{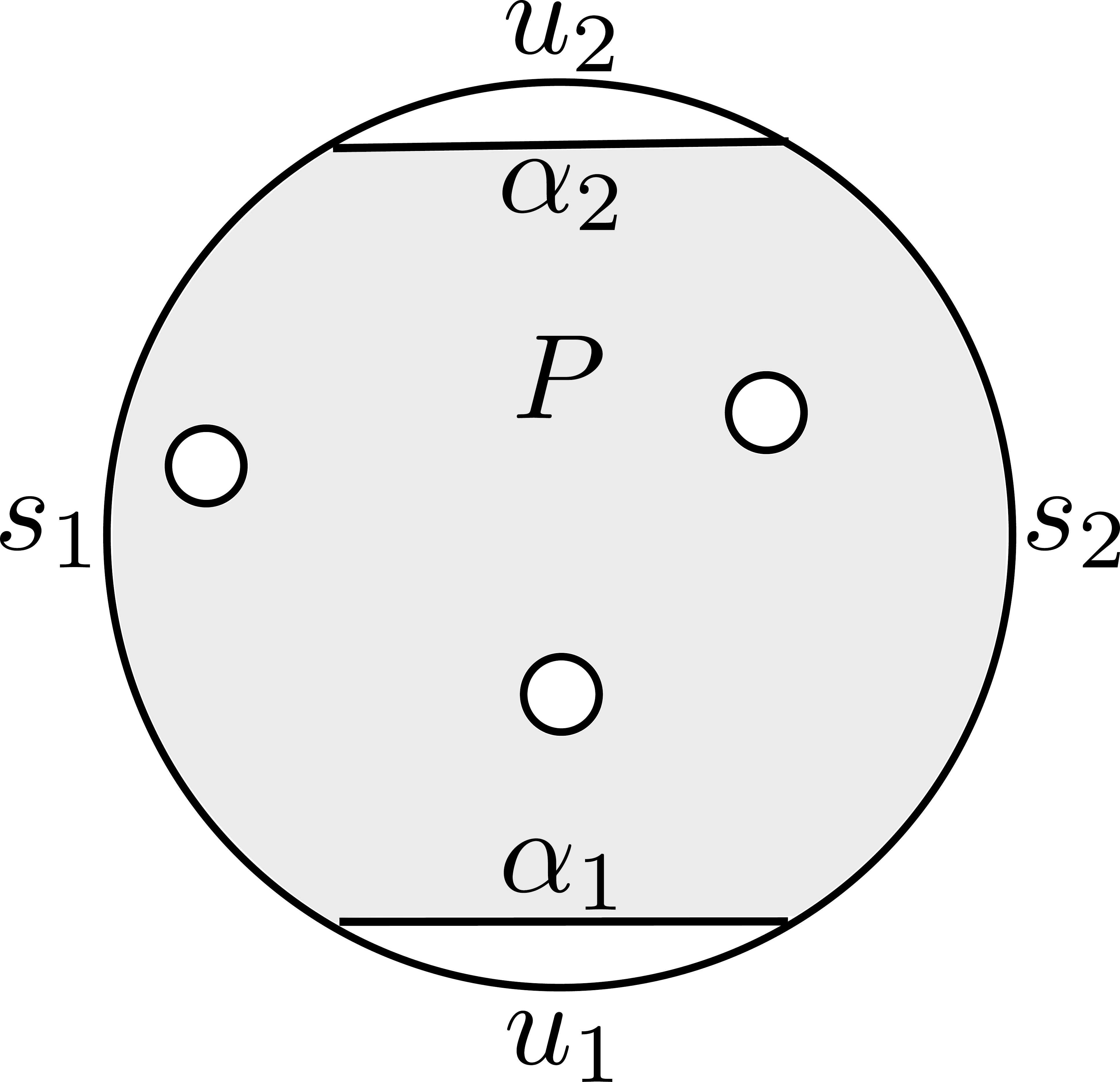} 
		\caption{If $\alpha_i$ connects  one end of $s_1$ to an end of $s_2$.}
		\label{figure:disk2}
	\end{figure}
	
	 Conversely, suppose that $s_1$ and $s_2$ are quasi-parallel and $P$ is a corresponding punctured disk. Let $K$ be the component of $\partial P$ containing $s_1\cup s_2$ and $u_1$ and $u_2$ be the two arcs of $K-(s_1\cup s_2)$. Isotope the interior of $u_1$ and $u_2$ into the exterior $B'$ of $B$ so that $\mathcal{T}'=(B',u_1\cup u_2)$ is a tangle. Then $\mathcal{T}\cup\mathcal{T}'$ is a tangle decomposition of $K$. Since each component of $\partial P-K$ bounds a disk in $B'-(u_1\cup u_2)$, then $K$ bounds a disk in $S^3$, hence it is trivial. Therefore $\mathcal{T}$ is unknottable.\end{proof}

\begin{theorem}
	A 2-string tangle is unlinkable if and only if its strings are quasi-trivial.
\end{theorem}
\begin{proof}
	Let $\mathcal{T}=(B, s_1\cup s_2)$ be a 2-string tangle.\\ 
	Suppose the strings of $\mathcal{T}$ are quasi-trivial and let $P_1$, $P_2$ be the corresponding disjoint punctured disks each co-bounds. Let $b_i$ be the boundary component of $P_i$ that contains $s_i$, for $i=1, 2$. Assuming $B$ in $S^3$, let $B'$ be the complement of $B$ and let $s_i'$ be the arc obtained by pushing the interior of $b_i-s_i$ into the interior of $B'$, so that $\mathcal{T}'=(B', s_1'\cup s_2')$ is a 2-string tangle. Let $L$ be the link in $S^3$ defined by $b_1\cup b_2$. Then $\mathcal{T}\cup\mathcal{T}'$ is a tangle decomposition of $L$. Since each component of $\partial P_i-b_i$ bounds a disk in $B'-(s_1'\cup s_2')$, we have that $b_1$ and $b_2$ bound disjoint disks. That is, $L$ is a trivial two component link and, therefore, $\mathcal{T}$ is unlinkable.\\ 
	Conversely, suppose that $\mathcal{T}$ is unlinkable with unlinking closure tangle $\mathcal{U}$. Let $b_1$ and $b_2$ be the components of $L$ and $O_1$, $O_2$ the corresponding disjoint disks they bound. We have that $b_i\cap B$ is $s_i$. Hence, $O_i\cap \partial B$ is a collection of simple closed curves and an arc $\alpha_i$ sharing the ends of $s_i$. Let $D_i$ be the (punctured) disk cut by $\alpha_i$ from $O_i$ that is co-bounded by $s_i$. Then, $s_1$ and $s_2$ are quasi-trivial strings. 
\end{proof}

\begin{theorem}
	A 2-string tangle is splittable if and only if it is quasi-inessential.
\end{theorem}
\begin{proof}
	Suppose that $\mathcal{T}$ is quasi-inessential and let $D$ be the punctured disk as in the definition. Let $s_i'$ be an arc in $\partial B$ connecting $\partial s_i$, disjoint from $\partial D$. Then the circle $s_1\cup s_1'$ together with the circle $s_2\cup s_2'$ define a 2-component link $L$. Let $B'$ be the exterior of $B$ in $S^3$ and $\mathcal{T}'$ the 2-string tangle defined by $(B', s_1'\cup s_2')$ after pushing the interior of $s_i'$ into the interior of $B'$. Each boundary component of $\partial D$ bounds a disk in $B'-(s_1'\cup s_2')$ resulting on a sphere $S$ that separates the components of $L$. Then, $\mathcal{T}\cup\mathcal{T}'$ is a tangle decomposition of a split link, that is, $\mathcal{T}$ is splittable.\\
	Conversely, suppose that $\mathcal{T}$ is splittable with splitting closure tangle $\mathcal{U}$. In case there is a disk in $B$ separating $s_1$ and $s_2$, then this disk makes $\mathcal{T}$ quasi-inessential, by definition. Hence, we can assume that $\mathcal{T}$ is essential. Let $S$ be the split sphere of $L$ and consider the intersection of $S$ with $B$. Suppose that $|S\cap \partial B|$ is minimal among all such spheres $S$. Note that it is necessarily non-empty as $B$ intersects both sides of $S$. As $\mathcal{T}$ is essential, the innermost curves of $S\cap \partial B$ in $\partial B$ have the corresponding innermost disks in $B'$. Suppose that some curve of $S\cap \partial B$ is not innermost in $S$. Let $c$ be innermost among these curves. Hence, $c$ cuts from $S$ a disk intersecting $\partial B$ only in innermost curves. Let $E$ be the corresponding punctured disk in $B$. The boundary components of $E$ are parallel in $\partial B-\partial (s_1\cup s_2)$ and bound disks in $B'-(s_1'\cup s_2')$. By capping off the boundary components of $E$ with those disks, either the resulting sphere splits the link $L$, and in this case we have a contradiction with $|S\cap \partial B|$ being minimal, or it bounds a ball in the exterior of $L$, and by cutting $S$ along $c$ and pasting a disk it bounds in $B'-(s_1'\cup s_2')$ we obtain a split sphere of $L$ that reduces $|S\cap \partial B|$, contradicting its minimality. Hence, all curves of $S\cap \partial B$ are innermost in $S$ with corresponding innermost disks in $B'$ separating $s_1'$ from $s_2'$. That is, $S$ intersects $B$ at a single punctured disk $D$ that separates $s_1$ from $s_2$ in $B$ and whose boundary components separate $\partial s_1$ from $\partial s_2$ in $\partial B$.\end{proof}


\section{Algebraic properties for 2-string tangles} \label{section: algebraic properties}

In this section we consider 2-string tangles $(B,\sigma)$ with the ends of $\sigma$ being four fixed points in the boundary circle $b$ of the diagram disk as in the intercardinal directions $NW$, $SW$, $NE$ and $SE$, illustrated in Figure \ref{figure:tangleends}. We say that a disk in $\partial B$ is a \textit{west disk} (resp., \textit{east disk}) if the disk intersects $b$ in a single arc containing $NW$ and $SW$ (resp., $NE$ and $SE$) but not any of the points $NE$ and $SE$ (resp., $NW$ and $SW$). Similarly we define a \textit{north disk} and a \textit{south disk} in $\partial B$.

\begin{figure}
	\centering
	\includegraphics[scale=.4]{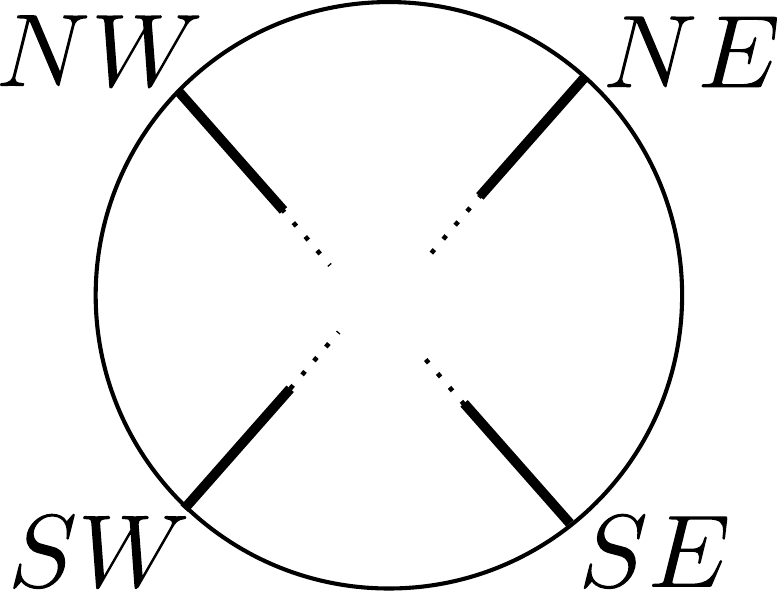}
	\caption{The ends of a tangle}
	\label{figure:tangleends}
\end{figure}

\begin{definition}
	We define the  \textit{sum} of two 2-string tangles $\mathcal{T}$ and $\mathcal{T}'$ by identifying an east disk in $\partial B$ to a west disk in $\partial B'$. The resulting tangle sum  is denoted $\mathcal{T+T'}$.\\
	We define the \textit{product} of $\mathcal{T}$ and $\mathcal{T}'$ by identifying a south disk of $\mathcal{T}$ to a north disk in $\mathcal{T}'$. The resulting tangle product is denoted $\mathcal{T*T'}$.\\
\end{definition}

Throughout this section we study unknottability, unlinkability and splittability of 2-string tangles up to strong equivalence and the behavior of these properties under tangle sum and product.

\begin{definition}
	Let $\mathcal{T}$ be a 2-string tangle. The {\em numerator closure} of $\mathcal{T}$ is the link $N(\mathcal{T})$ obtained by joining $NE$ and $NW$ by an unknotted arc and joining $SE$ and $SW$ by another unknotted arc;  the {\em denominator closure} of $\mathcal{T}$ is the link $D(\mathcal{T})$ obtained by joining $NW$ and $SW$ by an unknotted arc and joining $NE$ and $SE$ by another unknotted arc (See Figure \ref{figure:numerator_denominator}). 
\end{definition}

\begin{figure}
	\centering
	\includegraphics[scale=.3]{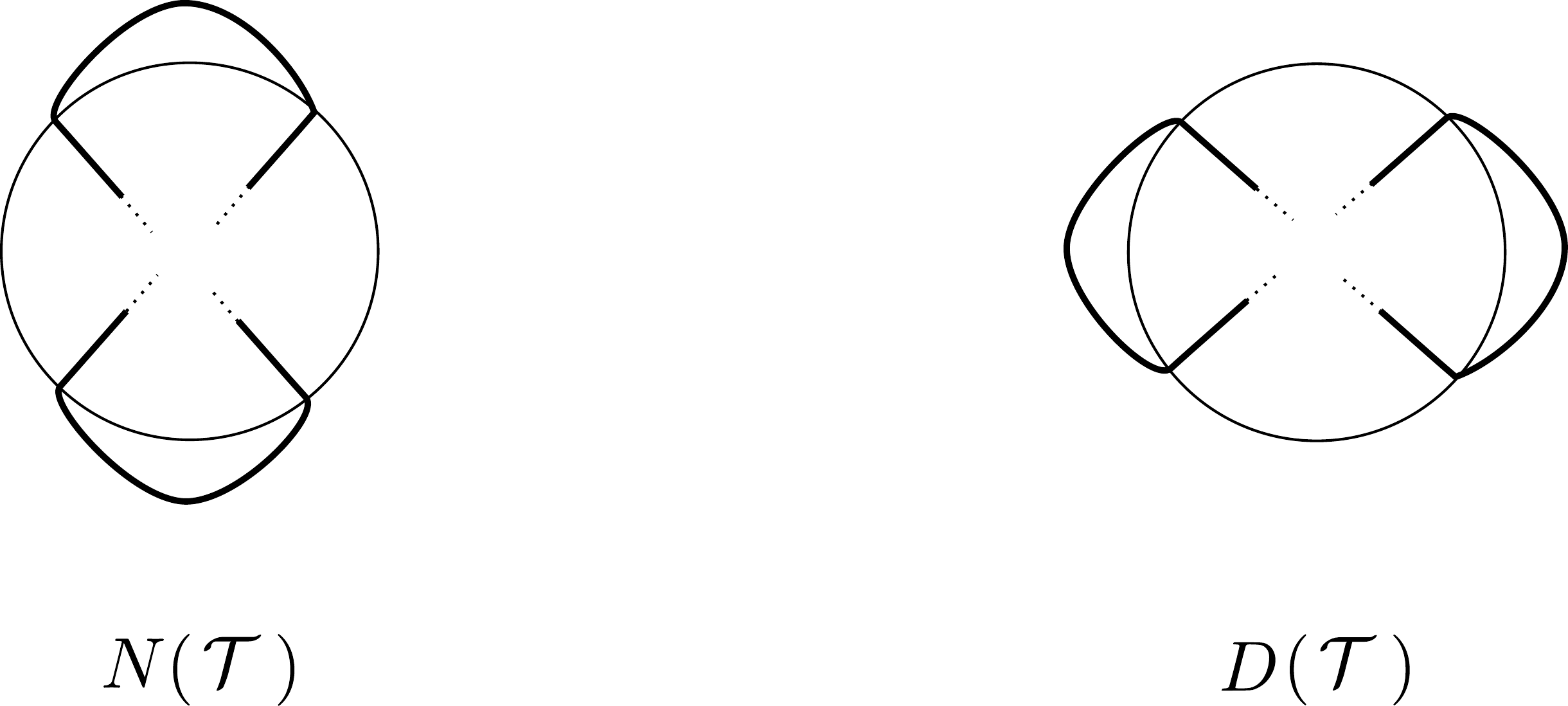}
	\caption{The numerator and denominator closures of a tangle.}
	\label{figure:numerator_denominator}
\end{figure}

\begin{definition}
	Let $\mathcal{T}$ be a $2$-string tangle.
	\begin{enumerate}
		\item The tangle obtained from $\mathcal{T}$ by a $180^\circ$ rotation around an axis perpendicular to the plane is denoted by $\bar{\mathcal{T}}$.
		\item The tangle obtained from $\mathcal{T}$ by a $90^\circ$ rotation around an axis perpendicular to the plane is denoted by $\mathcal{T}^\perp$.
	\end{enumerate}
\end{definition}

Note that a unknotting (resp. unlinking or splitting) closure tangle of a 2-string tangle $\mathcal{T}$ can be considered as a 2-string tangle $\mathcal{U}$ such that $N(\mathcal{T}+\mathcal{U})$ is the unknot (resp., the unlink or a split link). We have that $N({\mathcal T}+{\mathcal U})$ is equivalent to $D({\mathcal T}*\bar {\mathcal U})$ and for $\mathcal{U}$ a rational tangle, ${\mathcal U}=\bar {\mathcal U}$. Then, with $\mathcal{U}$ a rational tangle, $D({\mathcal T}*{\mathcal U})$ is also the unknot (resp., the unlink or a split link). From the next proposition, the unknotting or unlinking closure tangle of a 2-string tangle is a rational tangle, and we can assume that the splitting closure tangle of a 2-string tangle is a rational tangle.

\begin{proposition}\label{proposition:rational}
	Let $\mathcal{T}$ be an essential 2-string tangle. If $\mathcal{T}$ is unknottable (resp., unlinkable), then any unknotting (resp., unlinking) closure tangle is rational. 
	If $\mathcal{T}$ is splittable, then it has a rational splitting closure tangle.
\end{proposition}
\begin{proof}
	Suppose that $\mathcal{T}$ is unknottable or unlinkable with corresponding closure tangle $\mathcal{U}$, that is, $N(\mathcal{T}+\mathcal{U})$ is the unknot or the unlink, respectively. Therefore, $\mathcal{U}$ cannot have local knots and, as $\mathcal{T}$ is essential, $\mathcal{U}$ cannot be essential. Hence, $\mathcal{U}$ is trivial, that is, it is a rational tangle.\\
	Suppose now that $\mathcal{T}=(B, \sigma)$ is splittable with  splitting closure tangle $\mathcal{U}$, that is, $N(\mathcal{T}+\mathcal{U})$ is a split link. Let $S$ be a split sphere for $N(\mathcal{T}+\mathcal{U})$. Consider an innermost curve $c$ of $S\cap \partial B$. As $S$ is a sphere, $c$ bounds a disk $D$ in $S$, disjoint from the other components of $S\cap \partial B$. As $\mathcal{T}$ is essential, $D$ cannot be in $B$. Hence $D$ is in the exterior of $B$. That is, $\mathcal{U}$ is inessential. Therefore, $D$ separates the two strings of $\mathcal{U}$. Hence, in case both are unknotted, the tangle is trivial, and in case at least one is knotted, it can be replaced in the same ball separated by $D$ in $B'$ by an unknotted string. The resulting tangle $\mathcal{V}$ is a rational tangle. As the strings of $\mathcal{V}$ are disjoint from $S$, and the end points of the strings of $\mathcal{U}$ and $\mathcal{V}$ are the same, we have that $N(\mathcal{T}+\mathcal{V})$ is a split link. That is, $\mathcal{V}$ is a rational splitting closure tangle of $\mathcal{T}$.
\end{proof}

From Property P and double branched covers of unknottable 2-string tangles being knot exteriors in $S^3$ we have the following theorem, which is a consequence of work by Bleiler and Scharlemann in \cite{BS-86} and \cite{BS-88}.

\begin{theorem}[Bleiler and Scharlemann \cite{BS-86}, \cite{BS-88}]\label{theorem:unique_unknotting} 
	If an unknottable 2-string tangle is essential, then it has a unique unknotting closure tangle.
\end{theorem}

\begin{remark}
	If a 2-string tangle has two different unknotting closure tangles, then it is rational.
\end{remark}

A similar result was obtained by Eudave-Mu\~noz in \cite{EM-88} for splittability, and hence unlinkability, of 2-string tangles.

\begin{theorem}[Eudave-Mu\~noz \cite{EM-88}]\label{theorem:unique_split} 
	If a 2-string tangle is unlinkable (resp., splittable) then it has a unique unlinking (resp., rational splitting) closure tangle.
\end{theorem}

From work of Scharlemann \cite{S-85} we have that an essential 2-string tangle cannot be splittable and unknottable simultaneously.

\begin{theorem}[Scharlemann \cite{S-85}]\label{theorem:split_unknot}
	If a 2-string tangle is unknottable and splittable, then it is a rational tangle.
\end{theorem} 

In the next theorem, we determine when a Montesinos tangle is unknottable/unlinkable. This result extends, for instance, Theorem 5 of \cite{KL-11}, which gives all rational unknotting closures tangles for a rational tangle. From Theorem 2.2 of \cite{ES-90} we also obtain, in particular, the rational unknotting/unlinking closure tangles of a rational tangle.

\begin{theorem}\label{theorem:Montesinos}
	
	Let  ${\mathcal T}=\tangle{p_1}{q_1}+\tangle{p_2}{q_2}+\cdots+\tangle{p_n}{q_n}$, with $n\geq 2$ and $q_i> 1$. Then,
	\begin{enumerate}
		\item $\mathcal T$ is unknottable if and only if $n=2$ and $p_1q_2+p_2q_1\equiv\pm1\,(\!\!\!\!\mod q_1q_2)$.
		\item $\mathcal T$ is unlinkable if and only if $n=2$, $q_1=q_2$ and $p_1+p_2\equiv0\,(\!\!\!\!\mod q_1)$. 
		\item $\mathcal T$ is splittable if and only if it is unlinkable.
	\end{enumerate}
	Moreover, if  $\mathcal T$  is unknottable (resp. unlinkable, splittable), then its unknotting (resp. unlinking, rational splitting) closure tangle ${\mathcal U}$ is integral and, if $-1<\dfrac{p_1}{q_1},\dfrac{p_2}{q_2}<1$, then ${\mathcal U}$ is either $[0]$, $[1]$ or $[-1]$.
\end{theorem}

\begin{proof}
	Since ${\mathcal T}$ is not rational, then, by Proposition \ref{proposition:rational}, it has a rational unknotting/unlinking/splitting closure tangle ${\mathcal U}=\tangle{p}{q}$. The double cover of $S^3$ branched over  the Montesinos knot $N({\mathcal T}+{\mathcal U})$ is a Seifert fibered manifold $M$ \cite{M-73}, with invariant $$(0;p_1/q_1,\ldots,p_n/q_n,p/q).$$ By the classification of Seifert fibered spaces, if there are at least $3$ exceptional fibers, then $M$ is irreducible and not $S^3$. Therefore, if $\mathcal T$ is unknottable/unlinkable/splittable, then $n=2$ and $q=1$.
	\begin{enumerate}
		\item $\mathcal T$ is unknottable if and only if $M$ is $S^3$, which is equivalent to 
		$$\dfrac{p_1}{q_1}+\dfrac{p_2}{q_2}+p=\pm\dfrac{1}{q_1q_2},$$
		or, similarly, $p_1q_2+p_2q_1+pq_1q_2=\pm 1$. Moreover, if $-1<\dfrac{p_1}{q_1},\dfrac{p_2}{q_2}<1$, then clearly $-2<p<2$.
		\item $\mathcal T$ is unlinkable if and only if $M$ is $S^2\times S^1$, which is equivalent to  $$\dfrac{p_1}{q_1}+\dfrac{p_2}{q_2}+p=0,$$
		or, similarly, $q_1=q_2$ and $p_1+p_2+pq_1=0$. Moreover, if $-1<\dfrac{p_1}{q_1},\dfrac{p_2}{q_2}<1$, then clearly $-2<p<2$.
		\item $\mathcal T$ is splittable if and only if $M$ is reducible, which happens exactly when $M$ is $S^2\times S^1$, since $M$ is orientable and the fibration base is $S^2$.
	\end{enumerate}
\end{proof}

The following theorem describes the effect of adding a rational tangle to an unknottable/unlinkable/splittable 2-string tangle. 

\begin{theorem}\label{theorem:new_component}
	Let $\mathcal{T}$ be a 2-string tangle with unknotting closure tangle $\tangle{r}{s}$. Then, 
	\begin{enumerate}
		\item $\mathcal{T}+\tangle{p}{q}$ is unknottable if and only if $q=1$ or ($q=s$ and $p\equiv r\,(\!\!\!\!\mod q)$).
		\item $\mathcal{T}*\tangle{p}{q}$ is unknottable  if and only if  $p=1$ or ($p=r$ and $q\equiv s\,(\!\!\!\!\mod p))$.
	\end{enumerate} 
	
	A similar result holds for unlinkable/splittable tangles.
\end{theorem}

\begin{proof} Since the tangle $\mathcal{T}$ has unknotting closure tangle $\tangle{r}{s}$, then $N\left(\mathcal{T}+\tangle{r}{s}\right)=D\left(\mathcal{T}*\tangle{r}{s}\right)={\rm unknot}$.
	
	\begin{enumerate}
		\item Suppose that $\mathcal{T}+\tangle{p}{q}$ is unknottable, with unknotting closure tangle $U$. 
		
		Then  $N\left(\mathcal{T}+\tangle{p}{q}+ U\right)={\rm unknot}$, hence $\tangle{p}{q}+ U=\tangle{r}{s}$, by Theorem \ref{theorem:unique_unknotting}. Therefore, $U=[n]$ and $\tangle{p}{q}=\tangle{r}{s}+[-n]=\tangle{r-ns}{s}$. Since $(p,q)=(r,s)=1$, we conclude that $q=s$ and $p\equiv r\,(\!\!\!\!\mod q)$.
		\item  Since $D\left(\mathcal{T}*\tangle{r}{s}\right)={\rm unknot}$, then $N\left(\mathcal{T}^\perp+\tangle{-s}{r}\right)={\rm unknot}$.
		
		The tangle $\mathcal{T}*\tangle{p}{q}$ is unknottable if and only if $\mathcal{T}^\perp+\tangle{-q}{p}$ is unknottable.
		
		This is equivalent, by (a), to $p=r$ and $q\equiv s\,(\!\!\!\!\mod p)$.
\end{enumerate} \end{proof}

The following corollary characterizes unknottable, unlinkable and splittable tangles among algebraic tangles with three rational components.

\begin{corollary}\label{corollary:algebraic}
	The tangle $\left(\tangle{p_1}{q_1}+\tangle{p_2}{q_2}\right)*\tangle{p_3}{q_3}$, with $q_1,q_2,p_3\neq 1$, is unknottable if and only if $\tangle{p_1}{q_1}+\tangle{p_2}{q_2}$ is unknottable, with unknotting closure tangle $[p_3]$, and $q_3\equiv1\,(\!\!\!\!\mod p_3)$.
	
	A similar result holds for unlinkable/splittable tangles.
\end{corollary}

\begin{proof}
	Suppose that $\left(\tangle{p_1}{q_1}+\tangle{p_2}{q_2}\right)*\tangle{p_3}{q_3}$ is unknottable. By Theorem \ref{theorem:subtangle} Corollary \ref{corollary:subtangle}, $\tangle{p_1}{q_1}+\tangle{p_2}{q_2}$ is also unknottable, and, by Theorem \ref{theorem:Montesinos}, its unknotting closure tangle is an integral tangle $[n]$. Then, by Theorem \ref{theorem:new_component}(b), $\left(\tangle{p_1}{q_1}+\tangle{p_2}{q_2}\right)*\tangle{p_3}{q_3}$ is unknottable  if and only if $n=p_3$ and $q_3\equiv1\,(\!\!\!\!\mod p_3)$.
\end{proof}

\begin{theorem}
	
	Let $\mathcal{T}=(B, \sigma)$ be an essential 2-string tangle and $D$ a disk intersecting $\sigma$ and separating the ends of $\sigma$ in two sets of two points. Suppose that the tangles separated from $T$ by $D$ are essential. Then $\mathcal{T}$ is unknottable if and only if it has an unknotting closure  which is a rational tangle with the same slope as $\partial D$.
	
	A similar result holds for unlinkable/splittable tangles.

\end{theorem}

\begin{proof}
	
	If $\mathcal{T}$ has any unknotting closure tangle, then $\mathcal{T}$ is unknottable, by definition.\\
	Suppose now that $\mathcal{T}$ is unknottable, with unknotting closure tangle $\mathcal{U}$. Let $\mathcal{T}_i=(B_i, \sigma_i)$ denote the tangles separated by $D$ from $\mathcal{T}$, for $i=1, 2$. Consider the tangle $\mathcal{T}_2+\mathcal{U}$. As $\mathcal{T}_1$ is essential we have that $\mathcal{T}_2+\mathcal{U}$ is inessential. Let $E$ be a disk separating the strings of $\mathcal{T}_2+\mathcal{U}$. Note that $\partial E$ is in $\partial B_1$. As $\mathcal{T}_2$ is essential, in case $E$ intersects $\partial B_2$ we can reduce the number of components of $E\cap \partial B_2$ with an innermost curve/outermost arc argument, until $E$ is disjoint of $\partial B_2$.  Then, $\partial E$ is in $\partial B_1-D$ and separates the points of $D\cap \sigma_1$ from the other two end points of $\sigma_1$. Hence, the tangle $\mathcal{U}$, with respect to the fixed end points of $\mathcal{T}$, has the same slope as a rational tangle as $\partial D$ in $\partial B$.
\end{proof}

For a 2-string tangle $(B, \sigma)$, we say that a disk $D$ properly embedded in $B$ is a \textit{meridian} (resp., \textit{longitude}) {\em disk} if $\partial D$ separates $\partial B$ into a west and east disk (resp., into a north and south disk). In case $D$ is a meridian or longitude disk of $B$ with fixed boundary, we have the following corollary.

\begin{corollary}
	
	Let $\mathcal{T}=(B, \sigma)$ be an essential 2-string tangle and $D$ a meridian (resp., longitude) disk of $B$. Suppose that the tangles separated from $\mathcal{T}$ by $D$ are essential. Then $\mathcal{T}$ is unknottable if and only if $D(\mathcal{T})$ (resp., $N(\mathcal{T})$) is the unknot.
	
	A similar result holds for unlinkable/splittable tangles.
\end{corollary}

In particular, in the previous corollary, when $D$ intersects $\sigma$ at two points, we have the following corollary.

\begin{corollary}
	Let $\mathcal{T}_1$ and $\mathcal{T}_2$ be essential 2-string tangles. Then
	\begin{enumerate}
		\item $\mathcal{T}_1+\mathcal{T}_2$ is unknottable if and only if $D(\mathcal{T}_1+\mathcal{T}_2)={\rm unknot}$.
		\item $\mathcal{T}_1*\mathcal{T}_2$ is unknottable if and only if $N(\mathcal{T}_1*\mathcal{T}_2)={\rm unknot}$.
	\end{enumerate} 
	
	A similar result holds for unlinkable/splittable tangles.
	
\end{corollary}
\begin{proof}
	The sum of the two 2-string tangles corresponds to a decomposition through a meridian disk in the resulting tangle. Hence, (a) is an immediate consequence of the theorem. The statement (b) is equivalent to (a) by a $90^\circ$  rotation of the tangles.
\end{proof}

\section{Colorings}\label{section:colorings}
In this section we make further remarks on the coloring invariants of knots being an obstruction for a tangle to be unknottable. This observations follow in line with a remark by Silver in Krebes paper \cite{K} on Fox colorings being able to detect when a 2-string tangle is not unknottable, and with the work of Kauffman and Lopes in \cite{KL-20} on involutory quandle colorings, which include Fox colorings, being able to detect when a 2-string tangle is not unknottable. Here we make observations that these statements cannot be extended to oriented quandles. We also determine a necessary condition from coloring invariants for unlinkability and splittability of 2-string tangles.

\begin{definition}
	A {\em quandle} is a set $X$ with an operation $\triangleright$ such that
	\begin{itemize}
		\item $\forall x\in X,x\triangleright x=x$;
		\item $\forall y,z\in X,\exists! x\in X: z=x\triangleright y$; 
		\item $\forall x,y,z\in X,(x\triangleright y)\triangleright z=(x\triangleright z)\triangleright(y\triangleright z)$.
	\end{itemize} 
\end{definition}

\begin{definition}
	A {\em coloring} of an oriented knot or tangle by the quandle $(X,\triangleright)$ is a labeling of its arcs by elements of $X$ is such a way that at a crossing where the right underarc has label $x$, the overarc has label $y$, and the left underarc has label $z$, then $z=x\triangleright y$, as in Figure \ref{figure:crossing}.
	\begin{figure}[ht]
		\centering
		\includegraphics[scale=.4]{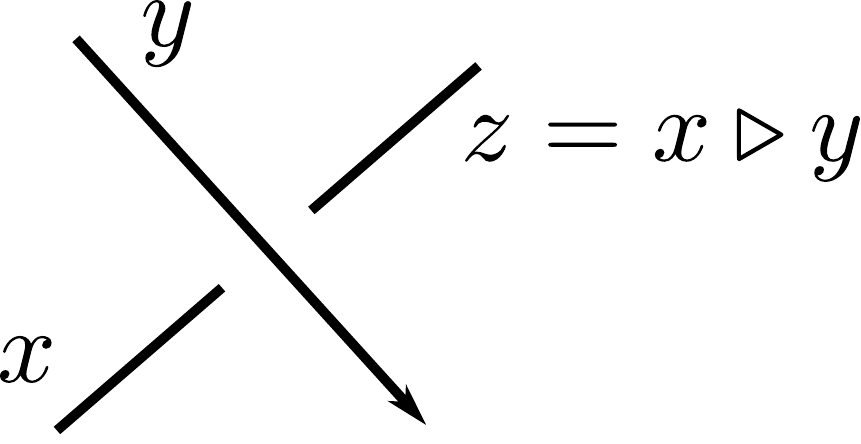}
		\caption{The consistency of the quandle operation and the labeling at each crossing.}
		\label{figure:crossing}
	\end{figure}
	
	We say that a quandle coloring of an oriented knot or tangle is {\em nontrivial} if it uses more than one color; we say that it is {\em trivial} otherwise. A tangle coloring such that all boundary arcs have the same color is called a $c$-coloring, and otherwise is called a $d$-coloring.
	We say also that an oriented knot is {\em polychromatic} if it has a nontrivial coloring and that an oriented tangle is {\em polychromatic} if it has a nontrivial $c$-coloring. Otherwise, the knot is said to be {\em monochromatic} and the tangle is said to be {\em monochromatic} if every $c$-coloring of the tangle is trivial.
	Since Reidemeister moves preserve  the existence of a nontrivial coloring, then an oriented knot or tangle is polychromatic if and only if any equivalent oriented knot or tangle is polychromatic. 
\end{definition}

If $(X,\triangleright)$ is a quandle, then $(X,\triangleleft)$, where $z\triangleleft y$ is the unique $x$ such that $z=x\triangleright y$, is also a quandle and a coloring of an oriented knot or tangle by $(X,\triangleright)$ determines a coloring of the knot or tangle with the opposite orientation by $(X,\triangleleft)$. However, if we change the orientation of some strings of a tangle while keeping the orientation of the remaining strings, it may happen that the original oriented tangle is polychromatic and the new oriented tangle is monochromatic. If the operations $\triangleleft$ and $\triangleright$ coincide, the quandle is called {\em involutory} or {\em unoriented}.
An oriented tangle has a nontrivial coloring by an involutory quandle  $(X,\diamond)$ if and only if it has a nontrivial coloring by $(X,\diamond)$, for any orientation of its strings. The {\em dihedral quandle} $R_n$ is the involutory quandle $\left(\mathbb{Z}/{n\mathbb{Z}},\diamond\right)$, where $x\diamond y\equiv 2y-x\,(\!\!\!\!\mod n)$ and $n$ is a  positive integer (for the sake of knot or tangle nontrivial colorings, we may consider only the case where $n$ is prime). We consider also the quandle $R_0=\left(\mathbb{Z},\diamond\right)$, where $x\diamond y=2y-x$. 


\begin{theorem}\label{theorem:monochromatic}
	If a tangle $\mathcal{T}$ is unknottable, then  $\mathcal{T}$ is monochromatic for some orientation.
\end{theorem}

\begin{proof}
	Let $\mathcal{U}$ be a unknotting closure tangle of $\mathcal{T}$. Suppose that $\mathcal{T}$ is polychromatic for the orientation induced by the unknot $K=N(\mathcal{T}+\mathcal{U})$. Then, by assigning the boundary color of $\mathcal{T}$ to all arcs of $\mathcal{U}$, we get a nontrivial coloring of $K$. 
	Since the trivial diagram of the unknot is monochromatic, we obtain a contradiction.
\end{proof}

The following corollary is also a remark by Silver in Krebes paper \cite{K} on Fox colorings being able to detect when a 2-string tangle is not unknottable, and it also appears in the work of Kauffman and Lopes in \cite{KL-20} on involutory quandle colorings being able to detect when a 2-string tangle is not unknottable.

\begin{corollary}\label{corollary:unknottablecoloring}
	If a tangle $\mathcal{T}$ is unknottable, then every $c$-coloring of $\mathcal{T}$ by an involutory quandle is trivial.
\end{corollary}

\begin{proof}
	Suppose that $\mathcal{T}$ has a nontrivial $c$-coloring by an involutory quandle. Then $\mathcal{T}$ is polychromatic for all orientations, hence, by Theorem \ref{theorem:monochromatic}, $\mathcal{T}$ is not unknottable.
\end{proof}

By Corollary \ref{corollary:unknottablecoloring}, if one can find a polychromatic coloring of a tangle by an involutory quandle, then this tangle is not unknottable. For instance, since the $2$-string tangles $6_2$, $6_3$, $7_{13}$, $7_{15}$, $7_{16}$, $7_{17}$ and $7_{18}$ have polychromatic colorings by dihedral quandles (see Figure \ref{figure:coloring} for one such coloring), they are not unknottable.\\

\begin{figure}[ht]
	\centering
	\includegraphics[scale=.12]{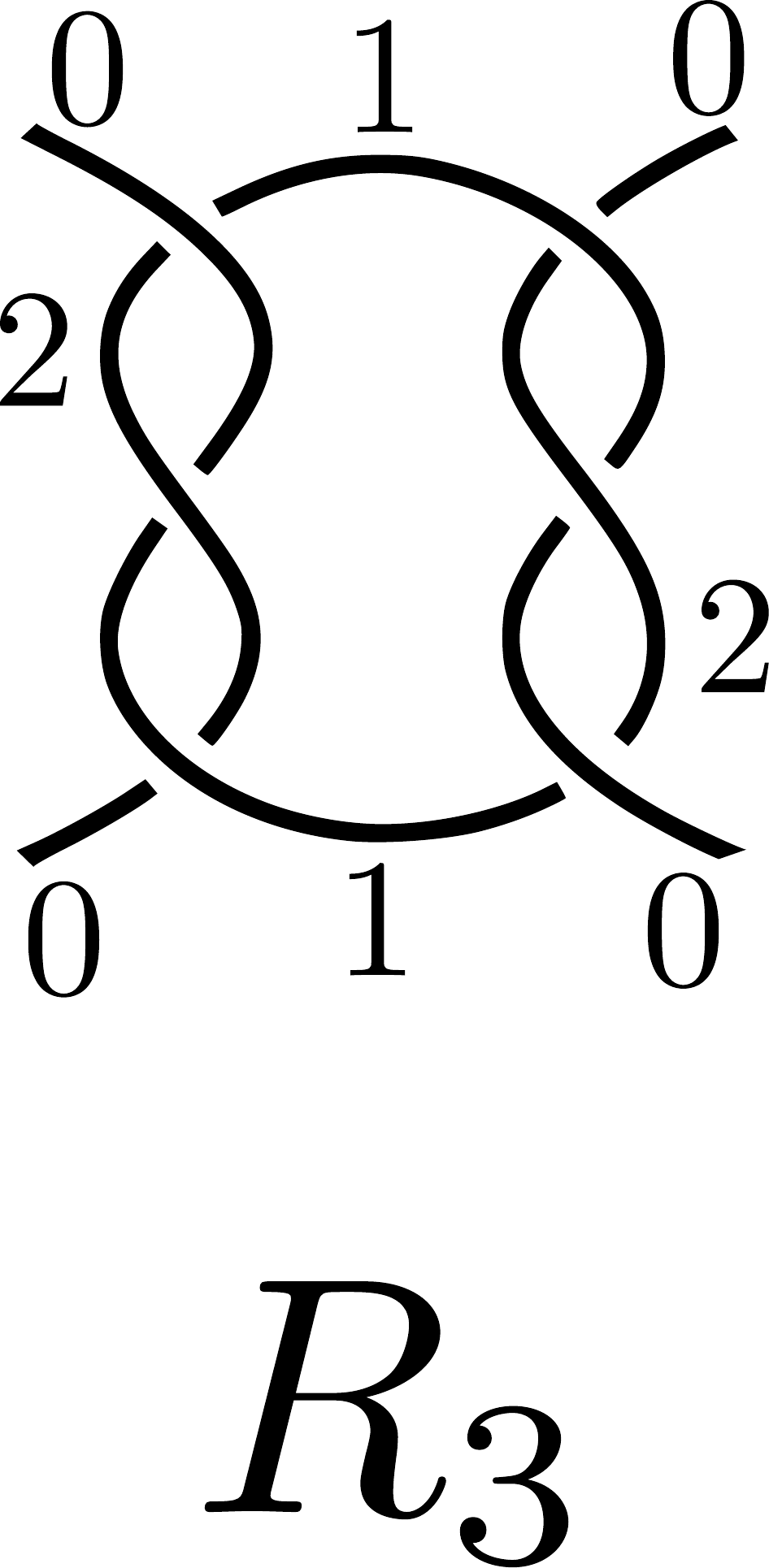}\quad		\includegraphics[scale=.12]{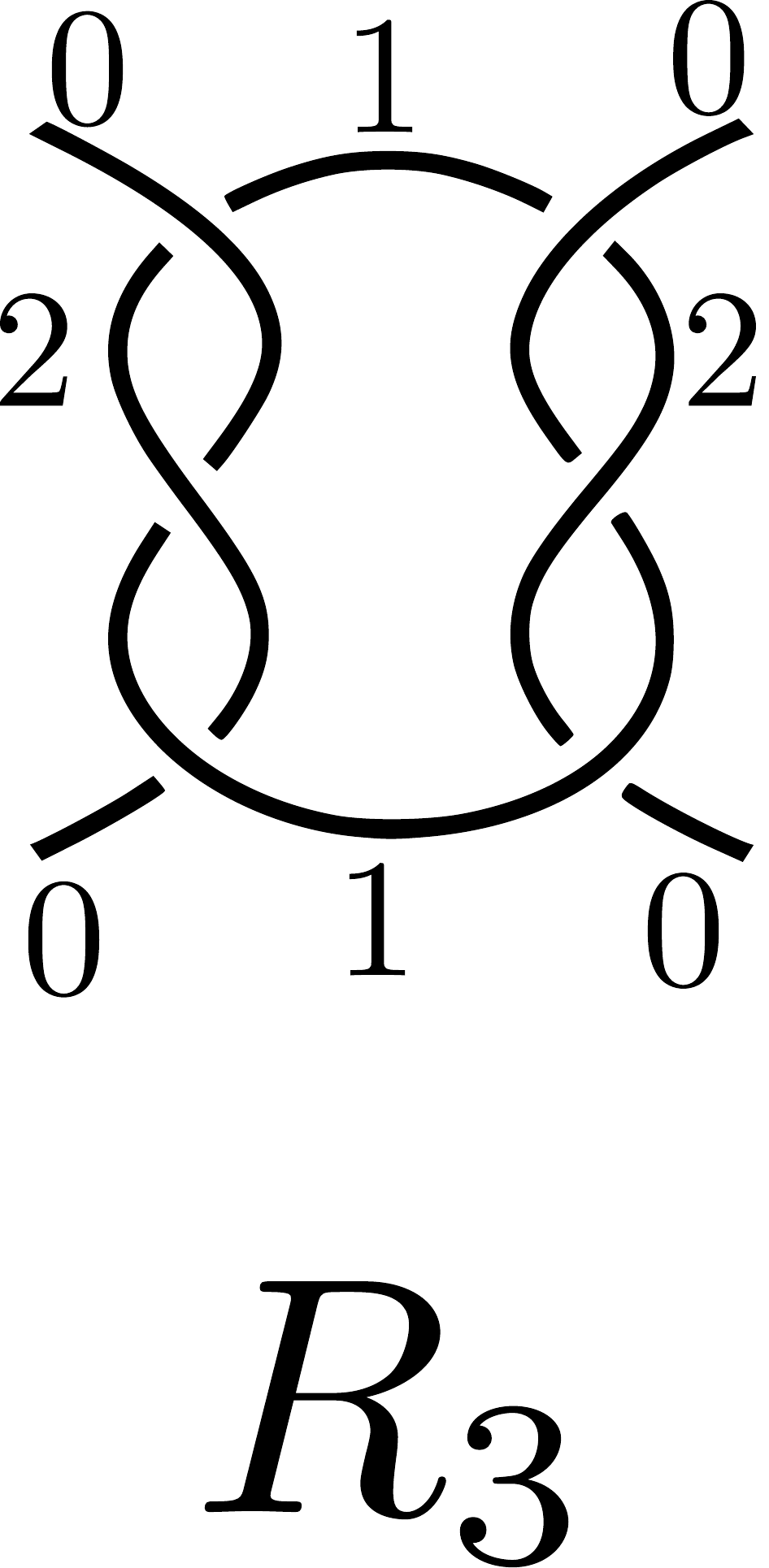}\quad		\includegraphics[scale=.12]{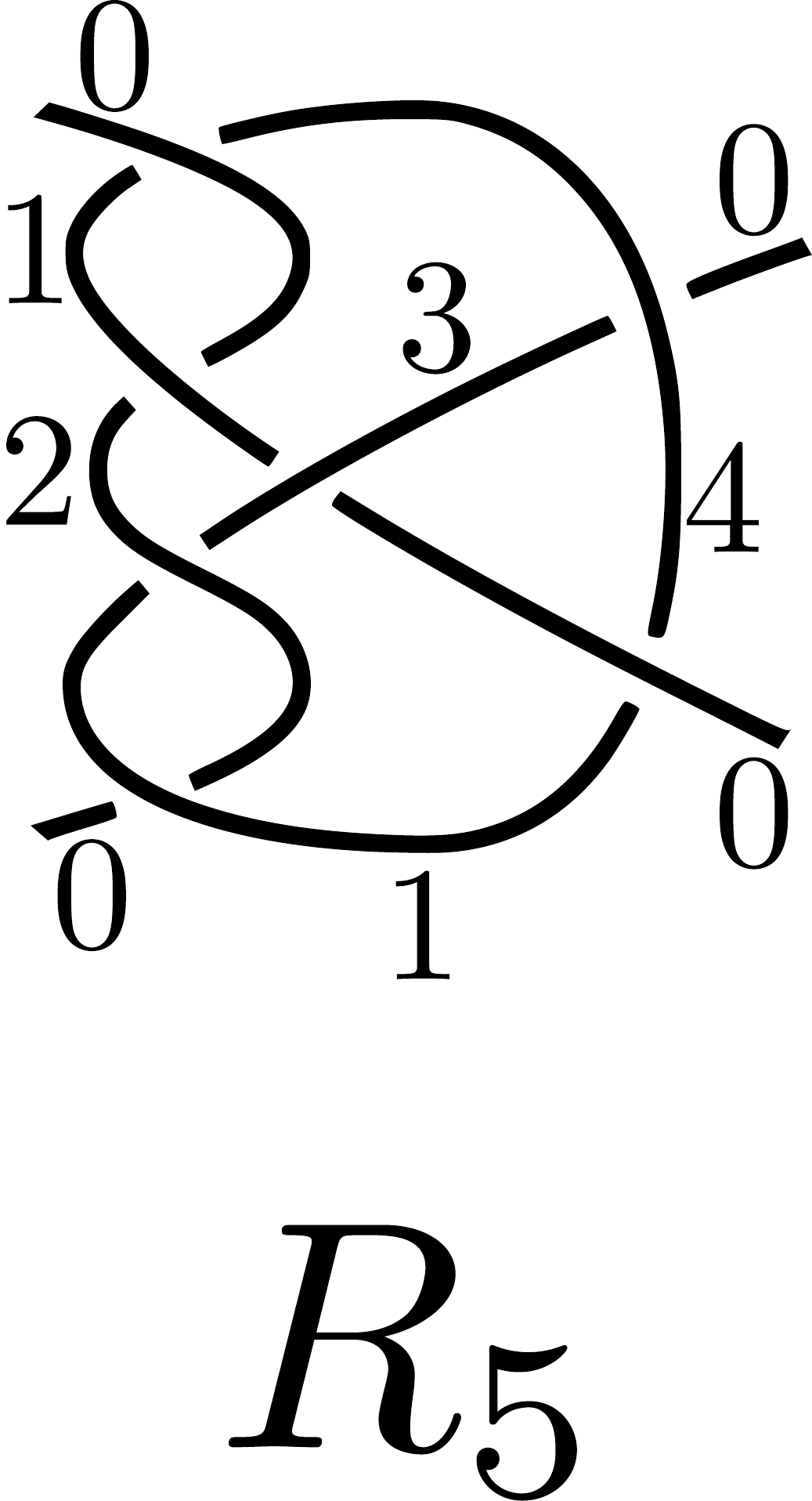}\quad		\includegraphics[scale=.12]{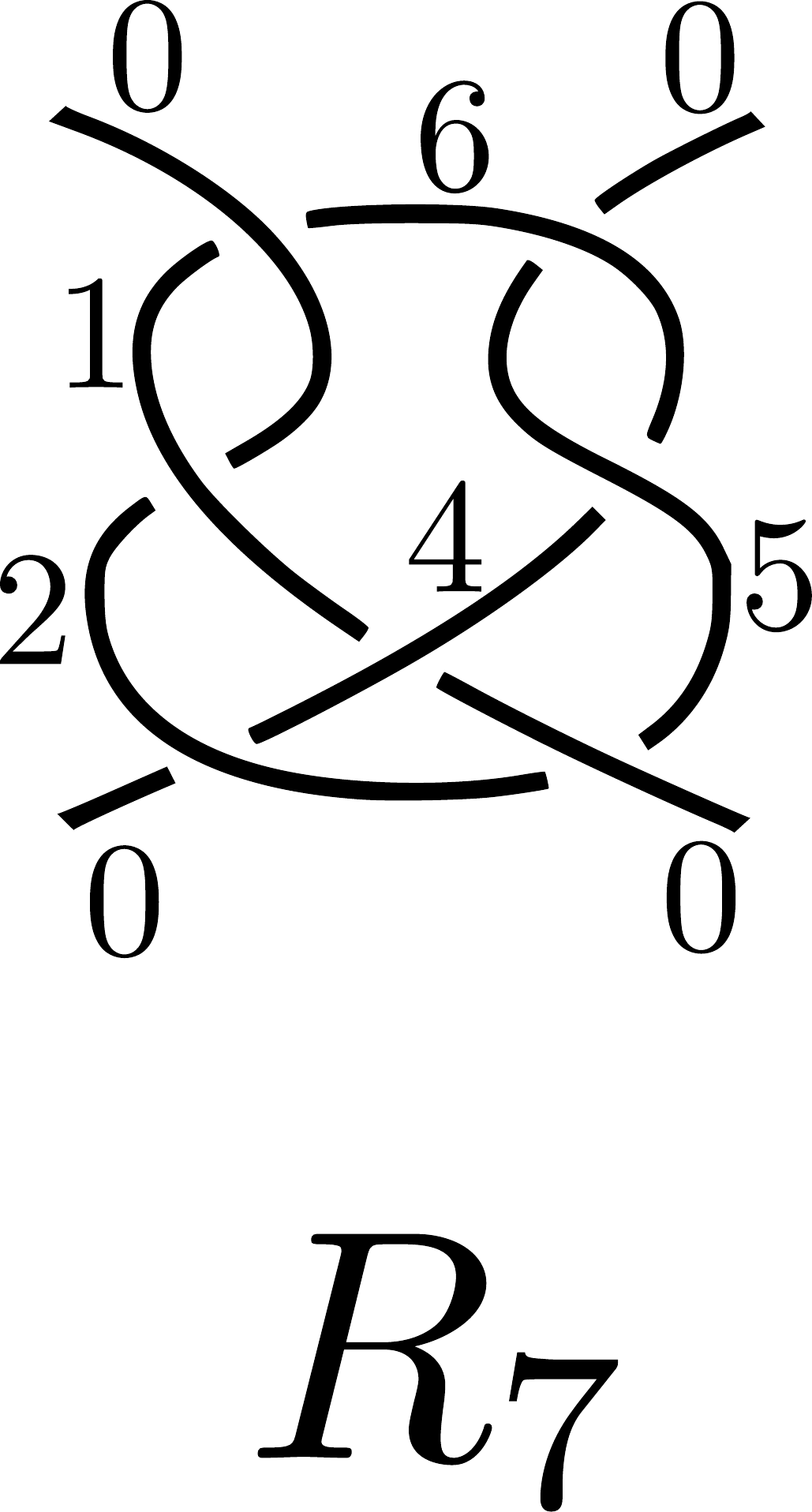}\quad 		\includegraphics[scale=.12]{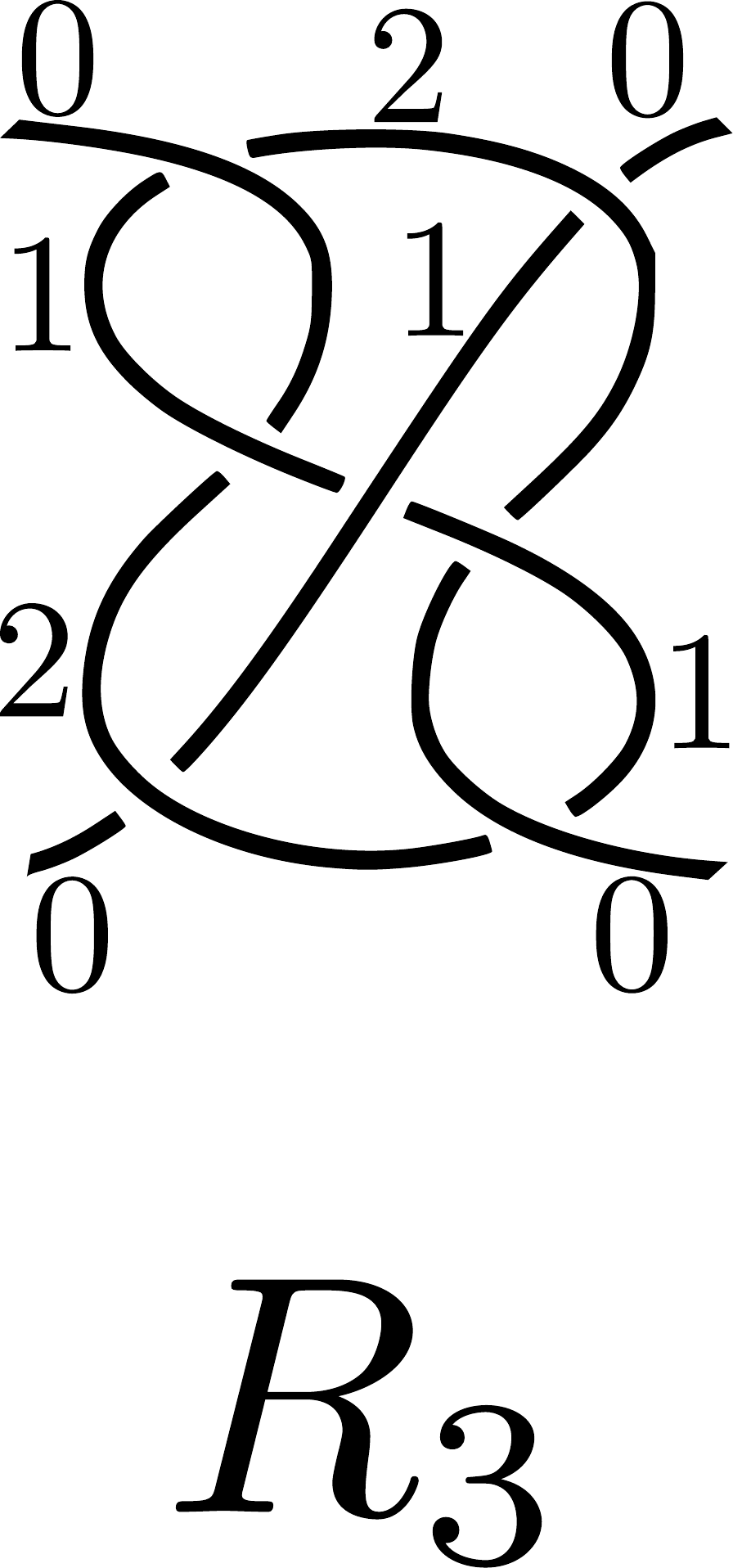}\quad		\includegraphics[scale=.12]{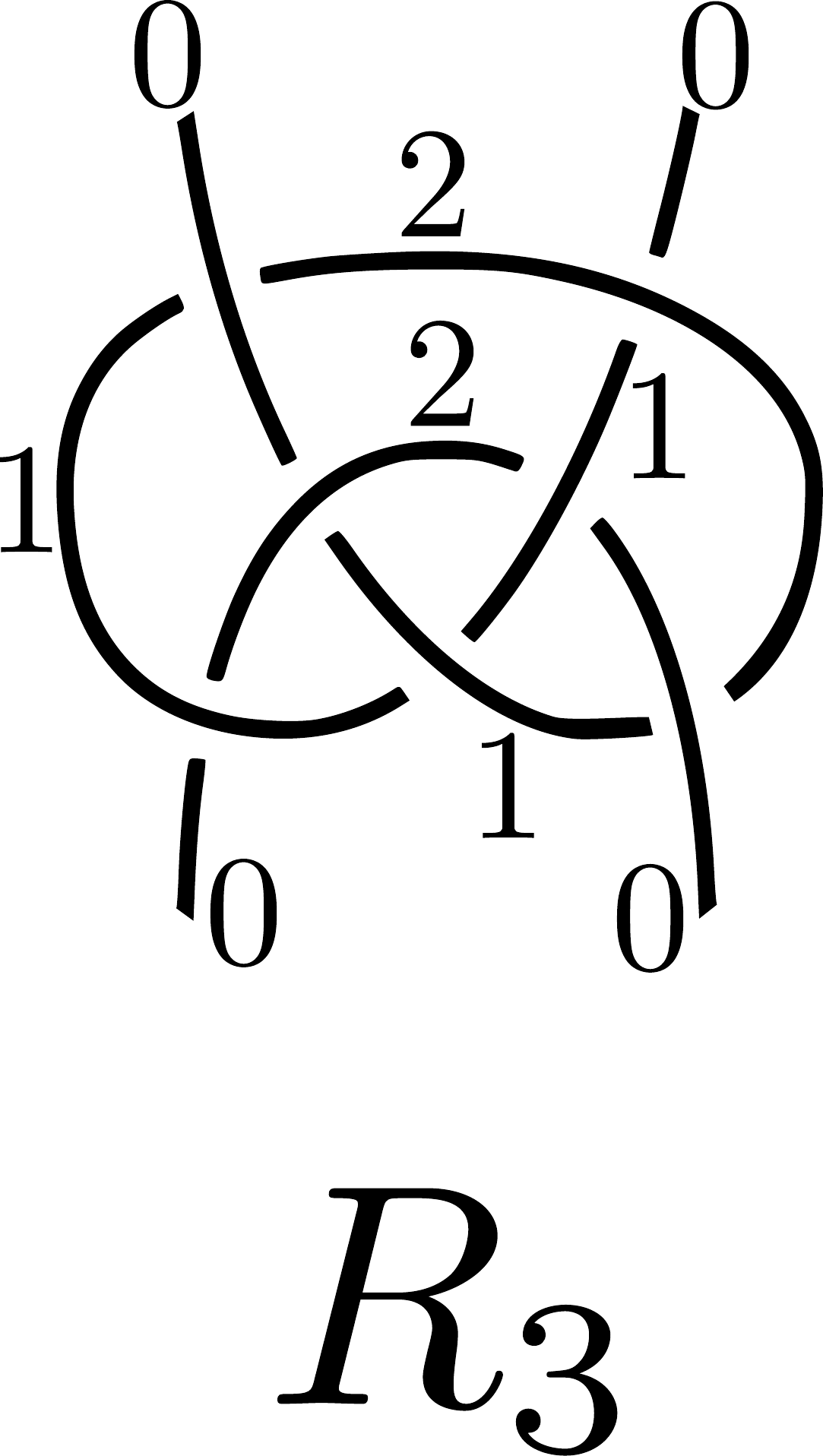}\quad
	\includegraphics[scale=.12]{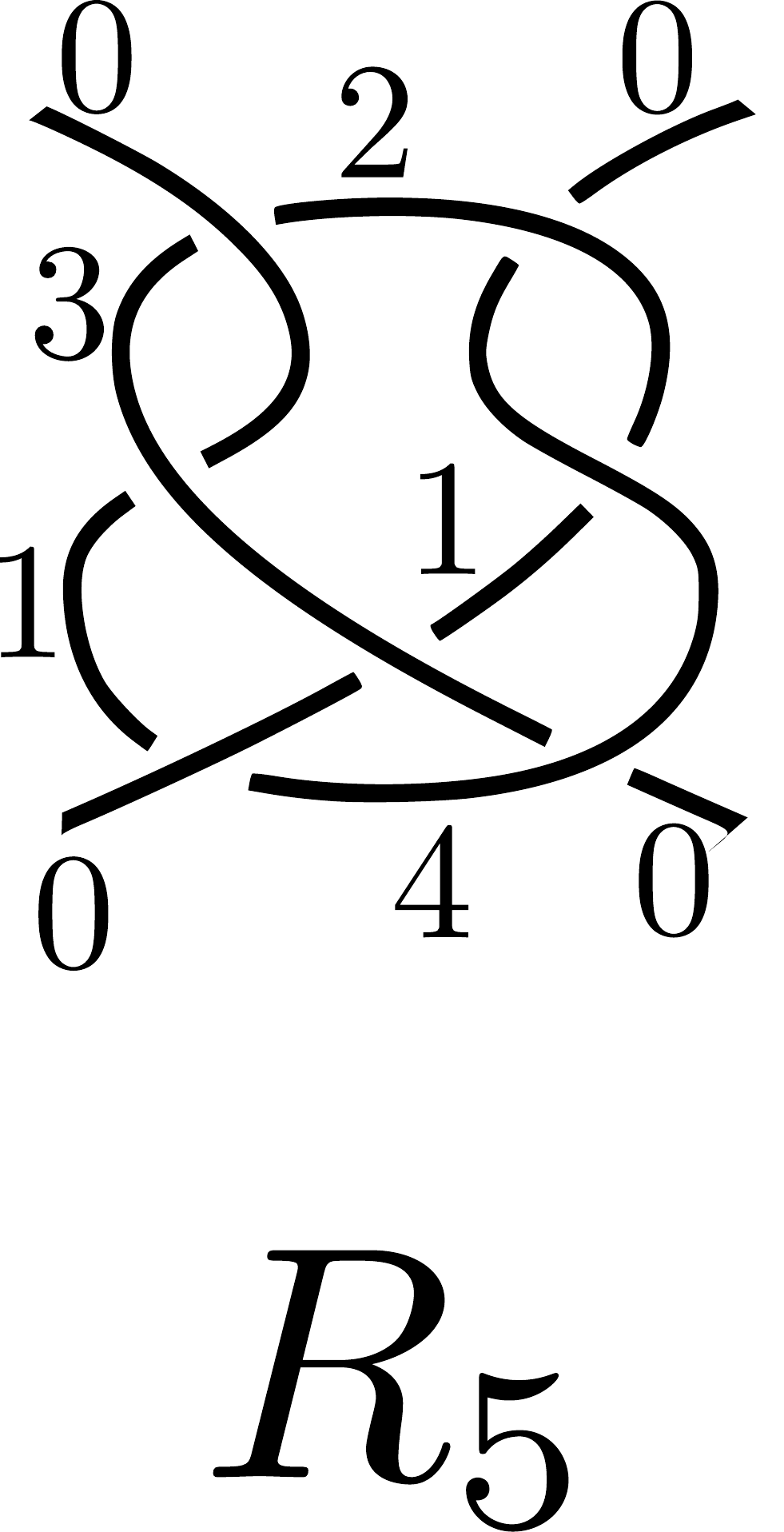} 
	\caption{Polychromatic colorings of  $6_2$, $6_3$, $7_{13}$, $7_{15}$, $7_{16}$, $7_{17}$ and $7_{18}$, respectively.}
	\label{figure:coloring}
\end{figure}





The converse of Theorem \ref{theorem:monochromatic} is not true. For instance, the tangle $6_4$ is not unknottable (see Section \ref{section:tables}), but it is monochromatic. To see this, suppose that there is a coloring of $6_4$ by some quandle, with the four ends having the same color $0$. Let the other arcs have colors $1$, $2$, $3$, $4$ as in Figure \ref{figure:notcolorable}. Since the arc with the color $1$ is in two crossings with a color $0$ overarc on the same side, then the colors $2$ and $3$ coincide. It then follows that the coloring is trivial.

\begin{figure}[ht]
	\centering
	\includegraphics[scale=.2]{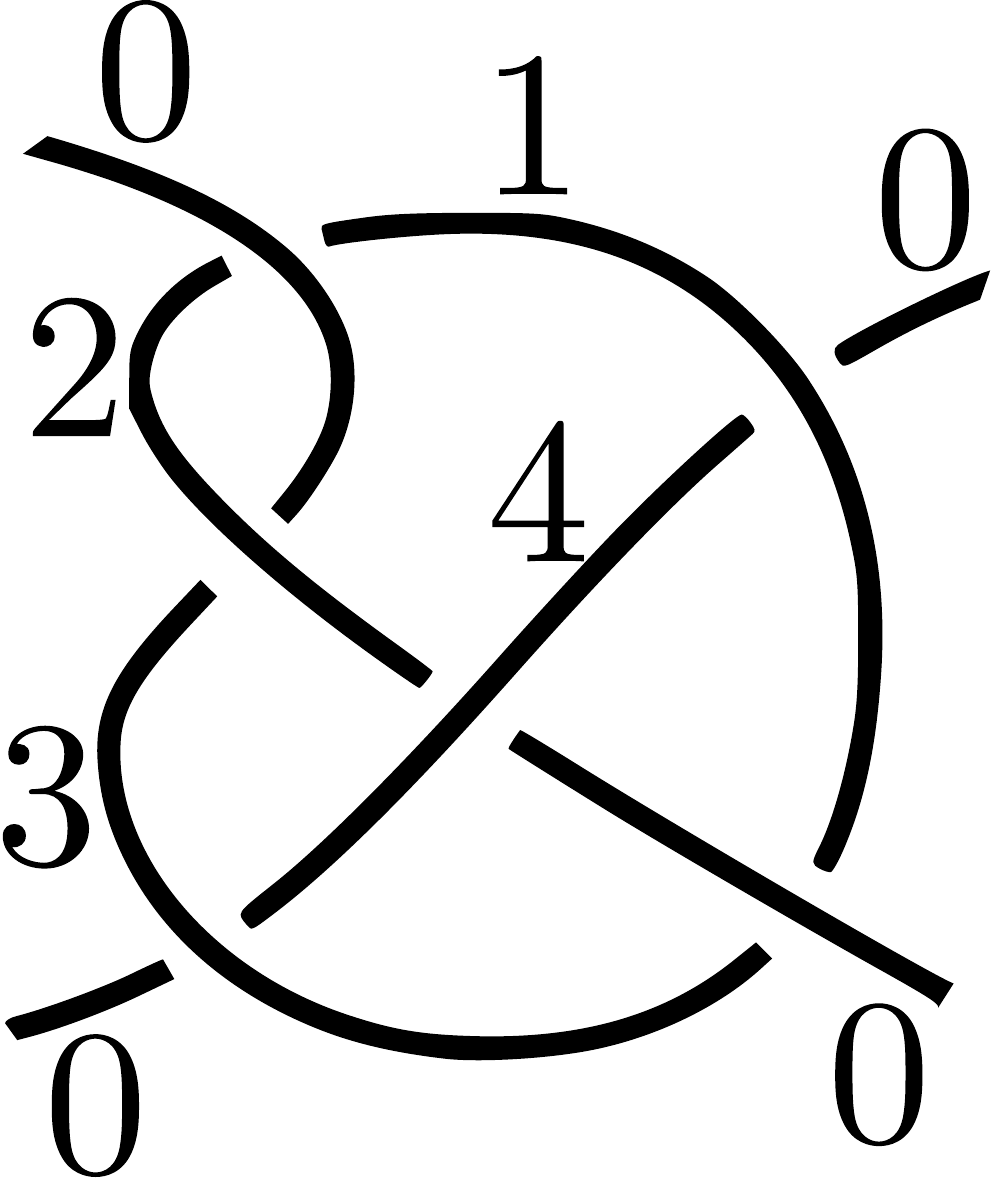}
	\caption{Any $c$-coloring of the tangle $6_4$ is trivial.}
	\label{figure:notcolorable}
\end{figure}

Also, the conclusion of Theorem \ref{theorem:monochromatic} cannot be extended to all orientations. For instance, the tangle $7_7$ is unknottable (see Section \ref{section:tables}),  but it is polychromatic for one of the orientations. 
To see this, consider the quandle ${\mathbb Z}_2[t]/(t^2+t+1)$ whose multiplication table is 

\begin{center}
	$\begin{array}{c|cccc}
		\triangleright&0&1&t&t+1\\
		\hline
		0&0&t+1&1&t\\
		1&t&1&t+1&0\\
		t&t+1&0&t&1\\
		t+1&1&t&0&t+1
	\end{array}$
\end{center}

There is a polychromatic coloring of one orientation of $7_7$ by this quandle (see Figure \ref{figure:77colorable}). The existence of this coloring implies that there is no  unknotting closure tangle of $7_7$ that connects the \textit{NW} and the \textit{SE} ends. 
\begin{figure}[ht]
	\centering
	\includegraphics[scale=.25]{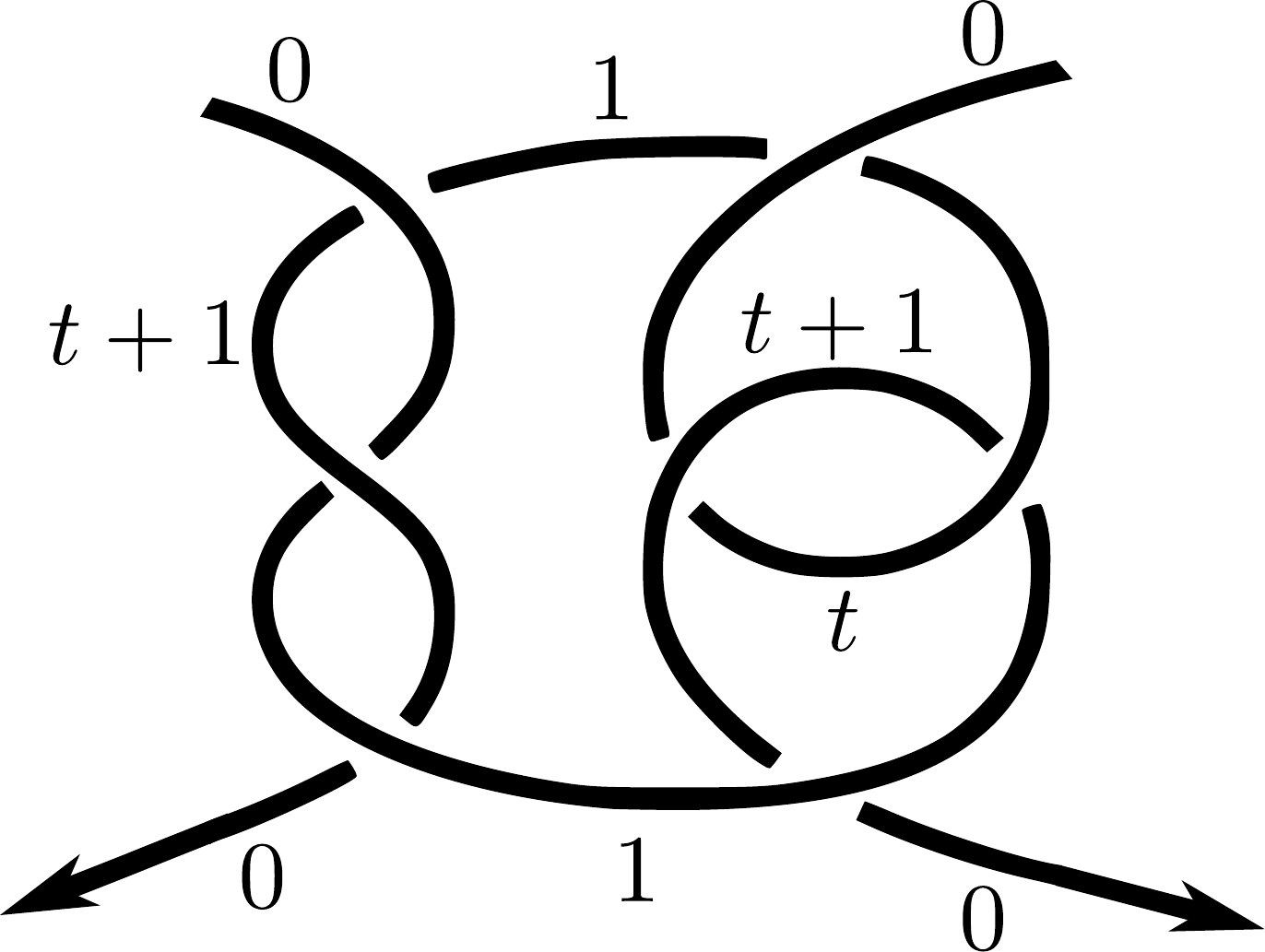}
	\caption{A coloring of the tangle $7_7$.}
	\label{figure:77colorable}
\end{figure}

\begin{theorem}\label{coloring}
	The 2-string tangle $\mathcal T$ has a polychromatic coloring by $R_n$ if and only if $\mathcal T+[\pm 2]$ has one such coloring. A similar result holds for $\mathcal T*[\pm 2]$.
\end{theorem}

\begin{proof}
	Any polychromatic coloring of $\mathcal T$ can be extended trivially to $\mathcal T+[\pm 2]$. Conversely, consider a $c$-coloring of $\mathcal T+[\pm 2]$. Then the four arcs of $[\pm2]$ must have the same color, hence the four boundary arcs of $\mathcal T$ also have that color. If the coloring of $\mathcal T+[\pm2]$ is nontrivial then the coloring of $\mathcal T$ is also nontrivial.
\end{proof}

This theorem shows, for instance, that $7_{16}$ has a polychromatic $R_3$-coloring, since $7_{16}\approx\mathcal{T}*[-2]$, where $\mathcal{T}\approx 6_2$, and the tangle $6_2$ has a polychromatic $R_3$-coloring.\\

To prove that a tangle is not unknottable we can use polychromatic  $R_n$-colorings of the tangle as this implies that any closure of the tangle would have more than $n$ distinct colorings, being this an obstruction for the closure to be the unknot. However, this is not sufficient to prove that a tangle is not unlinkable, as an unlink with $t$ components has $n^t$ distinct $R_n$-colorings. Hence, we could consider the possibility of proving that $\mathcal{T}$ is not unlinkable by showing that it has no non-trivial $R_n$-coloring for some $n$. The next theorem shows that this is also not a strategy to obstruct unlinkability.


The following theorems establish general results on $R_n$-colorings of tangles.

\begin{theorem}
	Every tangle $\mathcal{T}$ with more than one string has a nontrivial $R_n$-coloring, for every $n$.
\end{theorem}

\begin{proof}
	Let $\mathcal{T}$ be a tangle with $k$ crossings and $s$ strings. An $R_n$-coloring of $\mathcal{T}$ is determined by a system of $k$ linear equations (one for each crossing), with $k+s$ variables (one for each arc). The matrix defined by this system has rank at most $k$, thus has nullity at least $s$. Since $s>1$, there exists a nontrivial $R_n$-coloring of $\mathcal{T}$.
\end{proof}





We say that a $R_n$-coloring of a tangle verifies the {\em alternating  sum rule} if, ordering  the  ends  of  the tangle clockwise, the sum of the colors of the odd ends is the same as the sum of the colors of the even ends.  Note that a $R_n$-coloring of a 2-string tangle verifies  the alternating  sum rule if and only if the sum of colors of the \textit{NW} and \textit{SE} boundary arcs equals the sum of colors of the \textit{NE} and \textit{SW} boundary arcs. We say that a tangle $\mathcal{T}$ verifies the \textit{alternating sum rule} if every $R_n$-coloring of $\mathcal{T}$ verifies the alternating sum rule.

\begin{lemma}\label{lemma: alternating sum}
	Any tangle $\mathcal{T}$ verifies the alternating sum rule.
\end{lemma}

\begin{proof}
	The tangles $[0]$, $[1]$, $[-1]$ and $[\infty]$ clearly verify the alternating sum rule.
	
	Let $\mathcal{T}_1$ and $\mathcal{T}_2$ be tangles that verify the  alternating sum  rule. Consider the tangle $\mathcal{T}=\mathcal{T}_1\cup\mathcal{T}_2$ obtained by identifying some ends of $\mathcal{T}_1$ and $\mathcal{T}_2$. By ordering the ends of $\mathcal{T}_2$ in such a way that the common ends of $\mathcal{T}_1$ and $\mathcal{T}_2$ have opposite parities, we conclude that the difference of sum of the colors of the odd and even ends of $\mathcal{T}$ equals the sum of differences coming from $\mathcal{T}_1$ and $\mathcal{T}_2$, which is zero. Therefore, the sum of the colors of the odd ends of $\mathcal{T}$ is the same as the sum of the colors of the even ends.
\end{proof}

On the remainder of this section, we restrict ourselves to $R_0=\left(\mathbb{Z},\diamond\right)$ colorings. 

\begin{definition}
	Let $a,b,c,d$ be the \textit{NW}, \textit{NE}, \textit{SW}, \textit{SE} boundary colors of a $R_0$ $d$-coloring of a tangle $\mathcal{T}$. The {\em coloring fraction} of this coloring is defined by $\dfrac{b-a}{b-d}$.
\end{definition}

In the definition of coloring fraction, note that as the $R_0$ coloring is a $d$-coloring, from Lemma \ref{lemma: alternating sum}, $b-a$ and $b-d$ cannot be simultaneously 0. Hence,  the coloring fraction $\dfrac{b-a}{b-d}$ is well defined in $\mathbb{Q}\cup \{\pm \infty\}$.

The following theorem states that the coloring fraction of a rational tangle is the same as its arithmetical fraction.

\begin{theorem}[Kauffman and Lambropoulou \cite{KL-04}]The coloring fraction of a $R_0$ $d$-coloring of the rational tangle $\tangle{p}{q}$ is $\dfrac{p}{q}$.
\end{theorem}
\begin{proof}
	For $\tangle{p}{q}=[0]$ and $\tangle{p}{q}=[\infty]$, the result is immediate.\\
	Let $\mathcal{T}=\tangle{p}{q}$ be a tangle such that $\dfrac{b-a}{b-d}=\dfrac{p}{q}$. Then the boundary colors of $\mathcal{T}+[1]$ are $a,2b-d,c,b$. Since $\dfrac{2b-d-a}{2b-d-b}=\dfrac{b-a}{b-d}+1=\dfrac{p}{q}+1$, the result holds for the tangle $\mathcal{T}+[1]$.\\
	A similar reasoning (using the previous Lemma) shows that the result also holds for the tangles $\mathcal{T}+[-1]$, $\mathcal{T}*[1]$ and $\mathcal{T}*[-1]$. By successive applications of this property, the result holds for all rational tangles.\end{proof}

\begin{theorem}Consider a $R_0$-coloring of the tangle $\mathcal{T}+\mathcal{U}$ and its restrictions to $\mathcal{T}$ and $\mathcal{U}$. If these coloring are all $d$-colorings and the coloring fractions of  $\mathcal{T}$ and $\mathcal{U}$ are $r_1$ and $r_2$, then the coloring fraction of $\mathcal{T}+\mathcal{U}$ is $r_1+r_2$.
\end{theorem}
\begin{proof}
	Let $a,b,c,d$ be the boundary colors of $\mathcal{T}$ and $b,e,d,f$ be the boundary colors of $\mathcal{U}$. Then, since $a+d=b+c$ and $b+f=d+e$, we have
	$$
	\dfrac{e-a}{e-f}-\left(\dfrac{b-a}{b-d}+\dfrac{e-b}{e-f}\right)=0.
	$$\end{proof}

\begin{theorem}
	Let $L=N(\mathcal{T}+\mathcal{U})$ be a link with a nontrivial $R_0$-coloring. If, for $R_0$-colorings, $\mathcal{T}$ and $\mathcal{U}$ are monochromatic, then the coloring fractions of $\mathcal{T}$ and $\mathcal{U}$ are symmetric.
\end{theorem}
\begin{proof}
	If the coloring induced by $L$ on $\mathcal{T}$ and $\mathcal{U}$ was a $c$-coloring, then it would be trivial. Therefore, the induced colorings are $d$-colorings. If the induced coloring of $\mathcal{T}+\mathcal{U}$ is a $d$-coloring, then, by the previous theorem, the coloring fraction of $\mathcal{T}+\mathcal{U}$, which is 0 for the {\em NW} and {\em NE} colors being the same, is the sum of the coloring fractions of $\mathcal{T}$ and $\mathcal{U}$. If the induced coloring of $\mathcal{T}+\mathcal{U}$ is a $c$-coloring, then the coloring fraction of $\mathcal{T}$ is $\pm\infty$ and the coloring fraction of $\mathcal{U}$ is $\mp\infty$.
\end{proof}
\begin{corollary}\label{corollary: unlinking obstruction}
	Let $\mathcal{T}$ be a $R_0$-monochromatic tangle with unlinking (or rational splitting) closure tangle $\tangle{p}{q}$. Then $\mathcal{T}$ has coloring fraction $\dfrac{-p}{q}$.
\end{corollary}
\begin{proof}
	Every split link $N\left(\mathcal{T}+\tangle{p}{q}\right)$ has a nontrivial $R_0$-coloring and $\tangle{p}{q}$ is $R_0$-monochromatic.
\end{proof}

This corollary shows that a tangle with coloring invariant $p/q$ can possibly only have $\tangle{-p}{q}$ as its rational splitting closure tangle. However, even if the resulting link is not split, it has nevertheless zero determinant, as stated in the following theorem.
\begin{theorem}
	Let $\mathcal{T}$ be a tangle with coloring invariant $p/q$. Then $N\left(\mathcal{T}+\tangle{-p}{q}\right)$ is a link with determinant 0.
\end{theorem}
\begin{proof}
	Since there is a $d$-coloring of $\mathcal{T}+\tangle{-p}{q}$ such that its north ends of  have the same color, then $N\left(\mathcal{T}+\tangle{-p}{q}\right)$ has a nontrivial $R_0$-coloring. Therefore, for every prime number $n$ larger than the colors of $\mathcal{T}+\tangle{-p}{q}$, $N\left(\mathcal{T}+\tangle{-p}{q}\right)$ has a nontrivial $R_n$-coloring. Hence, $n$ divides the determinant of $N\left(\mathcal{T}+\tangle{-p}{q}\right)$, therefore this determinant must be  zero.
\end{proof}

\section{Unknottability of essential 2-string tangles with crossing number at most 7}\label{section:tables}

In this section, we classify unknottable  unlinkable and splittable tangles with at most $7$ crossings, with the notation of table 1 in \cite{KSS}.

\begin{theorem}\label{theorem:classification}An essential $2$-string tangle with crossing number at most 7 is unknottable if and only if it is equivalent to $5_1$, $6_1$, $7_2$, $7_5$, $7_7$ or $7_{14}$.
\end{theorem}
\begin{proof}
	It can be easily checked (see Figure \ref{figure:unknotting}) that 	$N(5_1+[-1])$, $N(6_1+[-1])$, $N(7_2+[-1])$, $N(7_5+[0])$, $N(7_7+[0])$ and $N(7_{14}+[-1])$ are unknotted.
	
	\begin{figure}[ht]
		\centering
		\includegraphics[scale=.08]{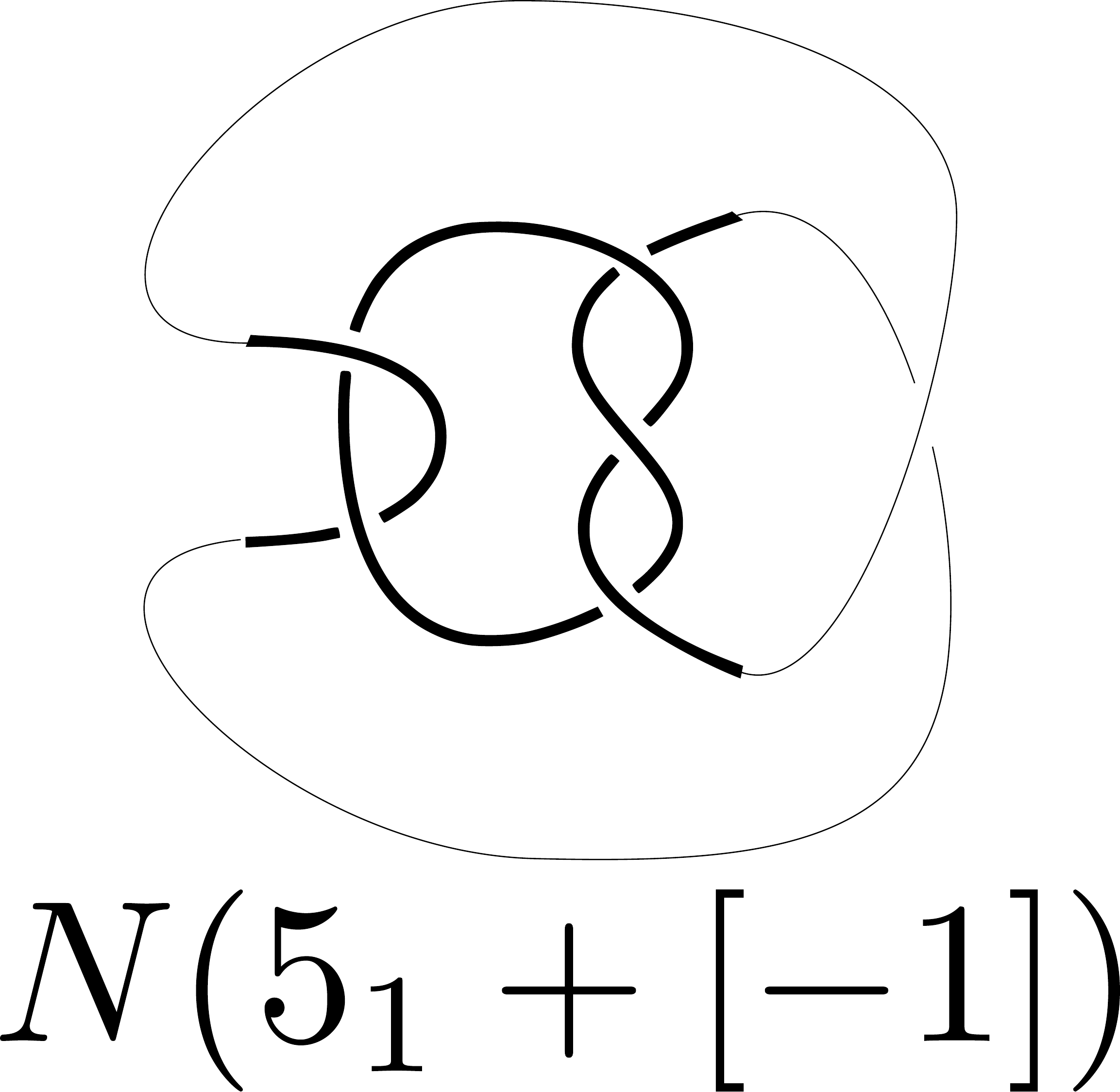}\,	\includegraphics[scale=.08]{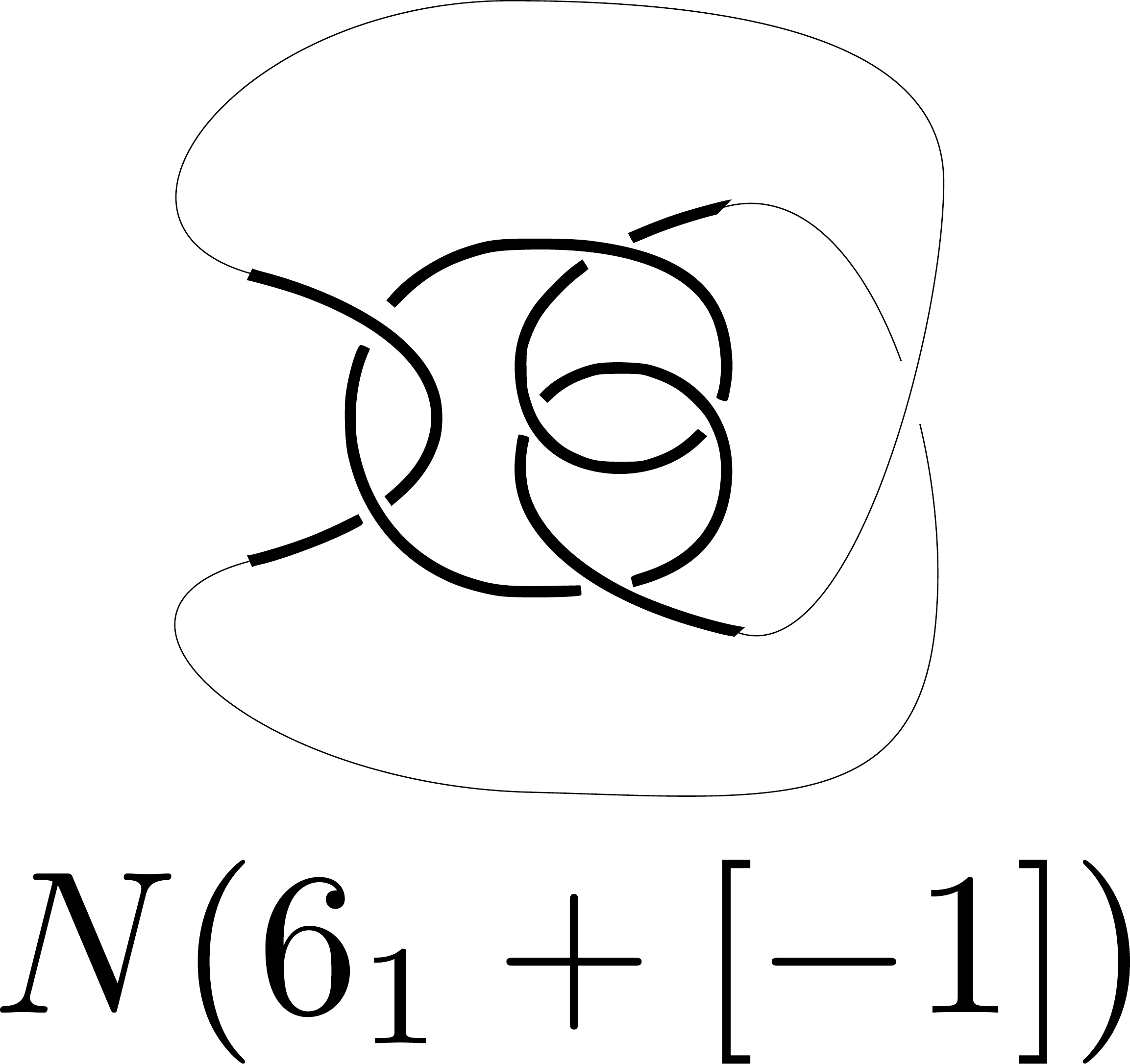}\, 	\includegraphics[scale=.08]{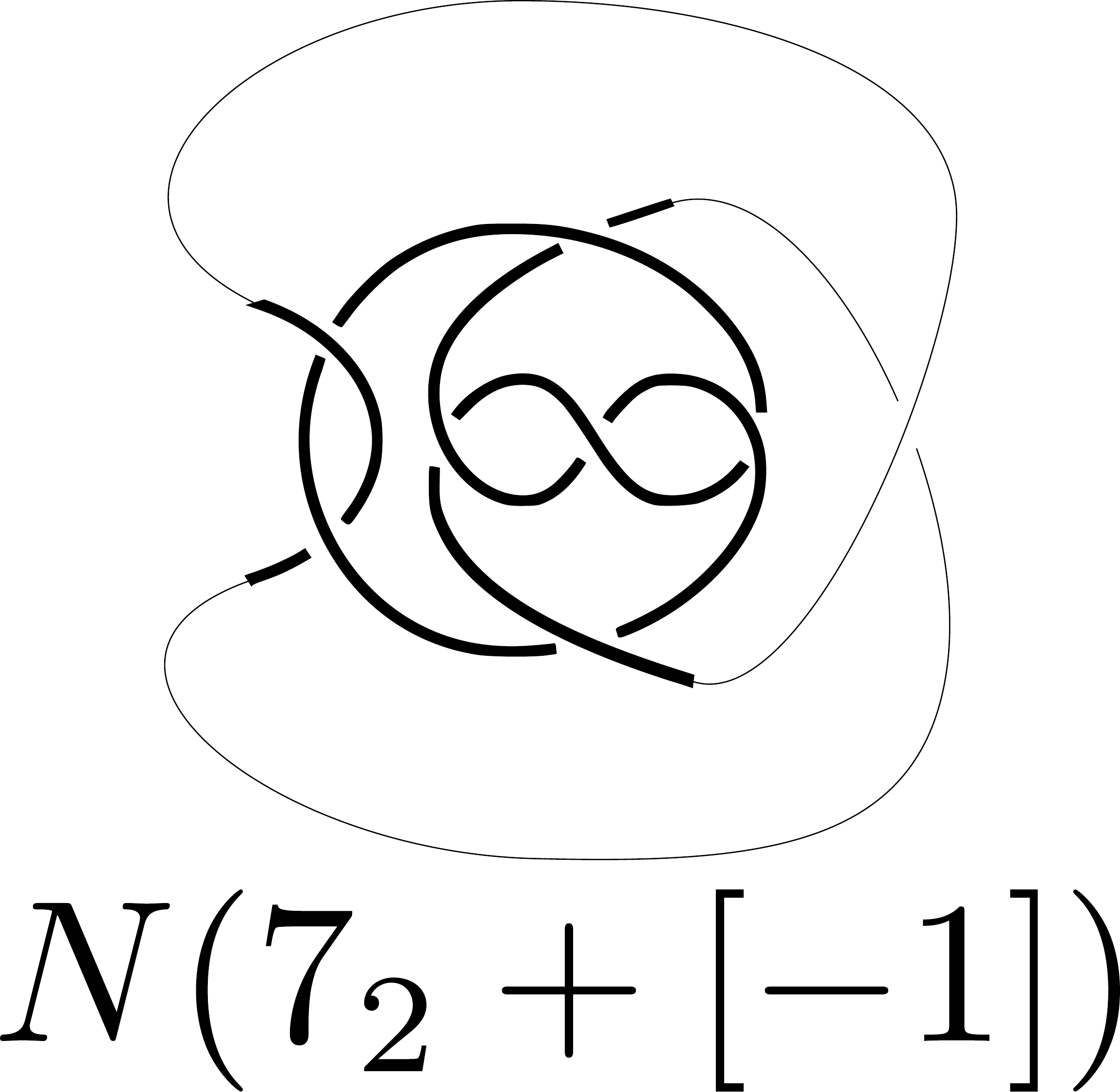}\,	\includegraphics[scale=.08]{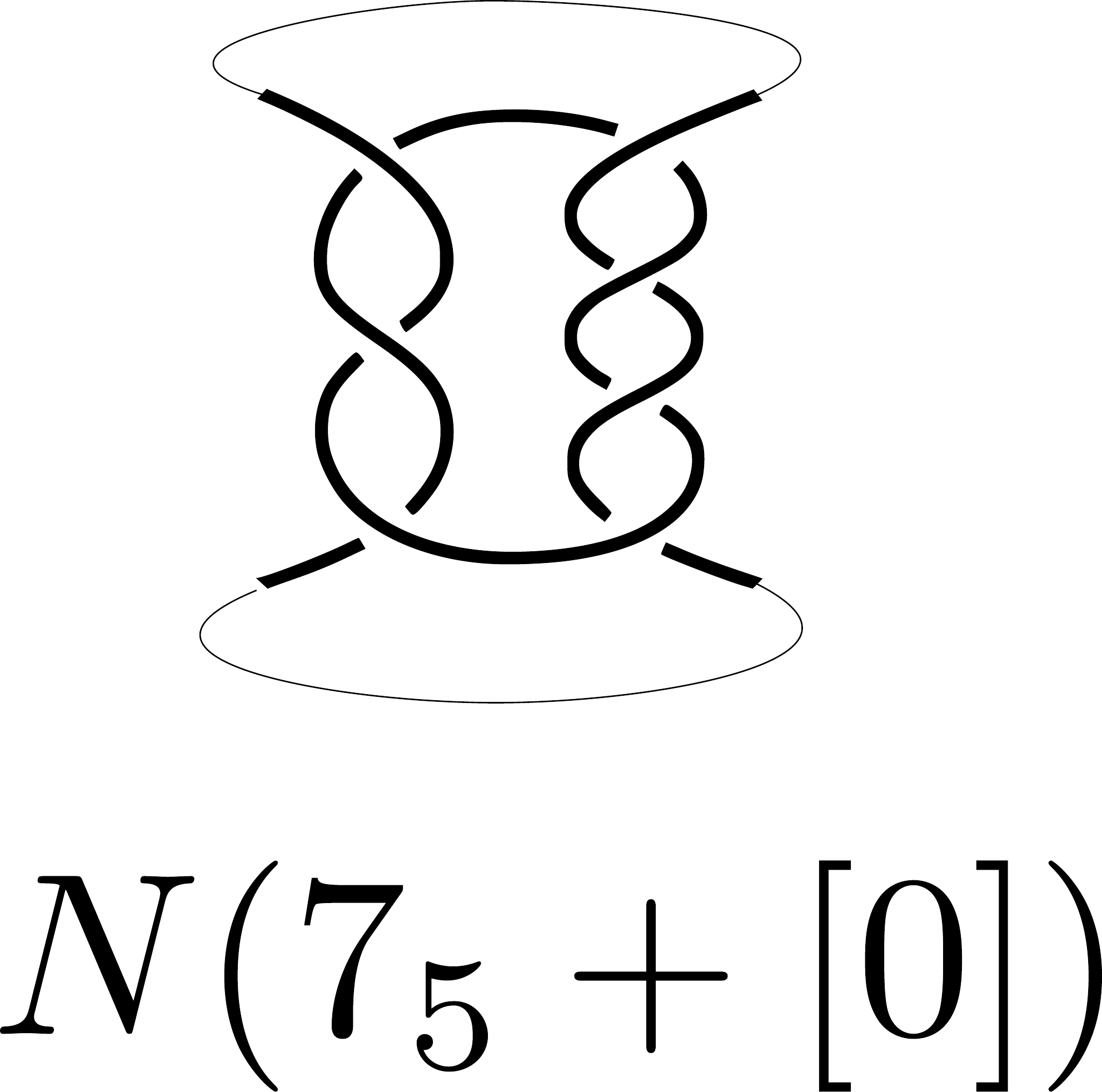}\, 	\includegraphics[scale=.08]{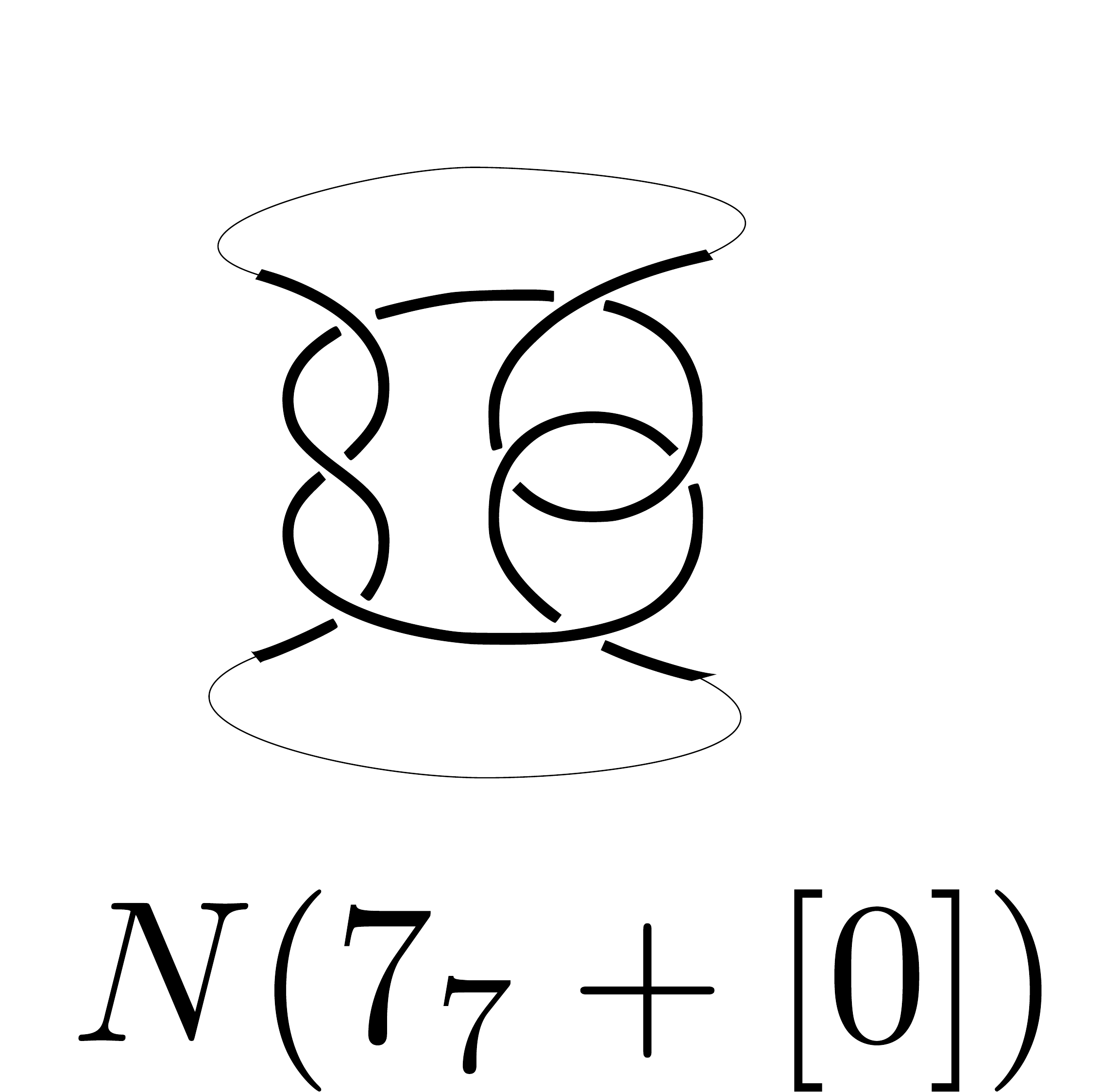}\, 	\includegraphics[scale=.08]{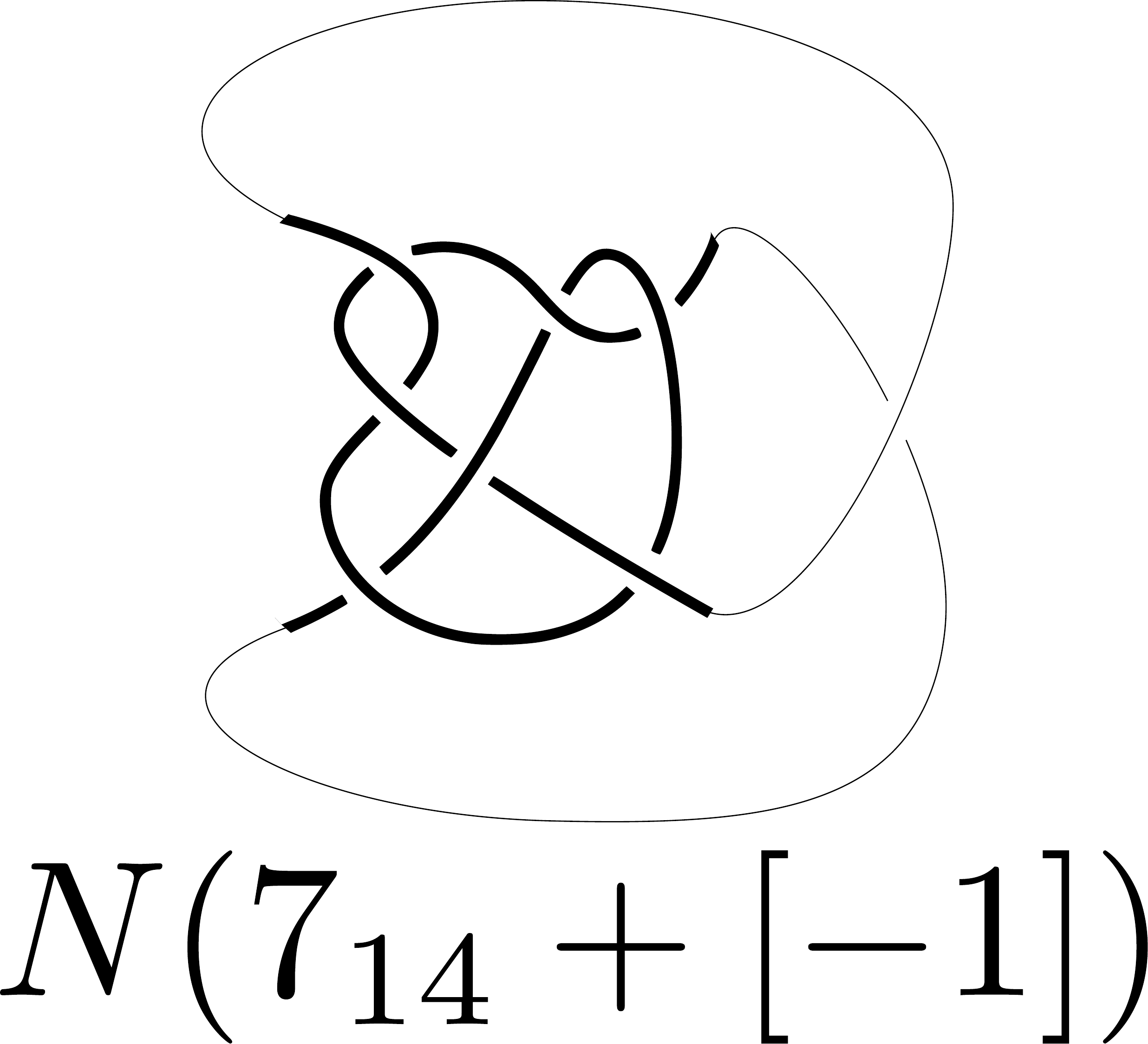}  
		\caption{The unknotting closure tangles of $5_1$, $6_1$, $7_2$, $7_5$, $7_7$ and $7_{14}$.}
		\label{figure:unknotting}
	\end{figure}

	 Conversely, the $2$-string tangles $6_2$, $6_3$, $7_{13}$, $7_{15}$, $7_{16}$, $7_{17}$ and $7_{18}$ have nontrivial colorings by dihedral quandles (see Section \ref{section:colorings}), therefore they are not unknottable. The remaining $2$-string tangles with crossing number at most $7$ are algebraic, with the expression given by Table \ref{table:algebraic}, and a direct application of Theorems \ref{theorem:Montesinos} and \ref{corollary:algebraic} shows that these tangles are not unknottable.
	
	\begin{table}[ht]
		\begin{center}
			\renewcommand{\arraystretch}{2}
			\begin{tabular}{c|cp{5mm}c|c}
				Tangle & Algebraic expression&&Tangle & Algebraic expression\\
				\cline{1-2}  \cline{4-5}
				$6_4$ &$\left(\tangle{1}{3}+\tangle{-1}{2}\right)*\tangle{-2}{1}$&&$7_8$&$\left(\tangle{-1}{2}+\tangle{-1}{3}\right)*\tangle{-2}{1}$ \\
				$7_1$ & $\tangle12+\tangle15$ &&$7_9$ & $\left(\tangle{-1}{2}+\tangle{-2}{3}\right)*\tangle{-2}{1}$ \\
				$7_3$ &$\tangle12+\tangle27$&&$7_{10}$&$\left(\tangle{-1}{2}+\tangle{1}{3}\right)*\tangle{3}{1}$ \\
				$7_4$ &$\tangle13+\tangle14$&&$7_{11}$&$\left(\tangle{-2}{3}+\tangle{1}{2}\right)*\tangle{-3}{1}$ \\
				$7_6$ &$\tangle13+\tangle25$&&$7_{12}$&$\left(\tangle{2}{3}+\tangle{-1}{2}\right)*\tangle{-2}{1}$ 
			\end{tabular}
		\end{center}
		\caption{Algebraic expressions of $6_4$, $7_1$, $7_3$, $7_4$, $7_6$, $7_8$, $7_9$, $7_{10}$, $7_{11}$ and $7_{12}$.}
		\label{table:algebraic}	
	\end{table}
\end{proof}

\begin{theorem}An essential $2$-string tangle with crossing number at most 7 is unlinkable or splittable if and only if it is equivalent to $6_3$.
\end{theorem}

\begin{proof}It can be easily checked (see Figure \ref{figure:unlinking}) that 	 $N(6_3+[0])$ is the unlink.
	\begin{figure}[ht]
		\centering
		\includegraphics[scale=.08]{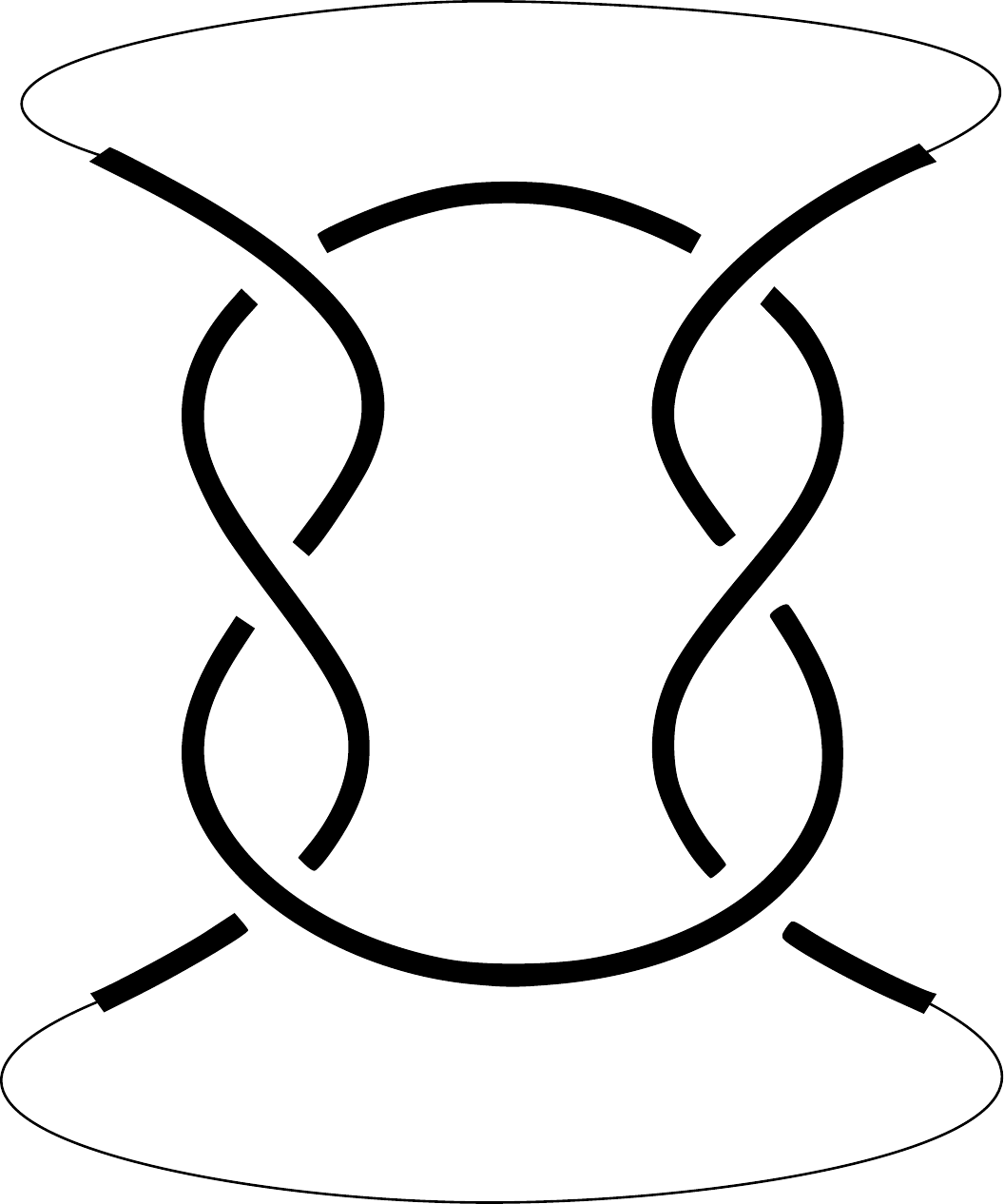}\,	
		\caption{The unlinking closure tangle of $6_3$.}
		\label{figure:unlinking}
	\end{figure}
	
	By Theorem \ref{theorem:split_unknot}, the tangles $5_1$, $6_1$, $7_2$, $7_5$, $7_7$ or $7_{14}$ are not unlinkable, since they are unknottable.
	
	The tangles $6_2$, $6_4$, $7_1$, $7_3$, $7_4$, $7_6$, $7_8$, $7_9$, $7_{10}$, $7_{11}$ and $7_{12}$ are algebraic, with the expression given by Tables \ref{table:algebraic} and \ref{table:algebraicII}, and a direct application of Theorems \ref{theorem:Montesinos} and \ref{corollary:algebraic} shows that these tangles are not unlinkable nor splittable.
	
	\begin{table}[ht]
		\begin{center}
			\renewcommand{\arraystretch}{2}
			\begin{tabular}{c|cp{5mm}c|c}
				Tangle & Algebraic expression&&Tangle & Algebraic expression\\
				\cline{1-2}  \cline{4-5}
				$6_2$ &$\tangle13+\tangle13$&&$7_{16}$ & $\left(\tangle13+\tangle13\right)*\tangle{-2}{1}$ \\

			\end{tabular}
		\end{center}
		\caption{Algebraic expressions of $6_2$ and $7_{16}$.}
		\label{table:algebraicII}
	\end{table}

	The tangles $7_{13}$, $7_{15}$, $7_{17}$ and  $7_{18}$ are $R_0$-monochromatic and have the coloring invariants of Table \ref{table:coloringinvariant}. Hence, from Corollary \ref{corollary: unlinking obstruction}, it remains to show that the corresponding rational tangles are not the unlinking or rational splitting closure tangles.
	
	\begin{table}[ht]
		\begin{center}
			\renewcommand{\arraystretch}{2}
			\begin{tabular}{c|cp{5mm}c|c}
				Tangle & Coloring invariant&&Tangle & Coloring invariant\\
				\cline{1-2}  \cline{4-5}
				$7_{13}$ &\tiny{$\dfrac{3}{4}$}&&$7_{17}$ &\tiny{$\dfrac{8}{7}$} \\
				$7_{15}$ &\tiny{$\dfrac{2}{3}$}&&$7_{18}$ & \tiny{$2$}\\
			\end{tabular}
		\end{center}
		\caption{Coloring invariants of $7_{13}$, $7_{15}$, $7_{17}$ and $7_{18}$.}
		\label{table:coloringinvariant}
	\end{table}
	
	The link $N(7_{13}+\tangle{-3}{4})$, as in Figure \ref{figure:not_splittable}, is not split since its Jones polynomial, $-t^{-3}+t^{-1}-t-t^5-t^7+t^9-t^{11} + t^{13}$, is different from the Jones polynomial of a split link defined by the right-handed trefoil and the unknot, $t^2+t^6-t^8$. 
	The link $N(7_{15}+\tangle{-2}{3})$, as in Figure \ref{figure:not_splittable}, is not split  since its Jones polynomial, 	$t^{-11}-t^{-9}-t^{-5}-t + t^3-t^5$, is different from the Jones polynomial of the two component unlink, $-t-t^{-1}$.  
	The links $N(7_{17}+\tangle{-8}{7})$ and  $N(7_{18}+[-2])$, as in Figure \ref{figure:not_splittable}, are not split since their linking number is not zero. Therefore, from Corollary \ref{corollary: unlinking obstruction}, the tangles $7_{13}, 7_{15}, 7_{17} \text{ and } 7_{18}$ are not splittable, and hence also not unlinkable (we can also conclude that the tangle $7_{13}$ is not unlinkable since one of its strings is knotted).	
	\begin{figure}[ht]
		\centering
		\includegraphics[scale=.08]{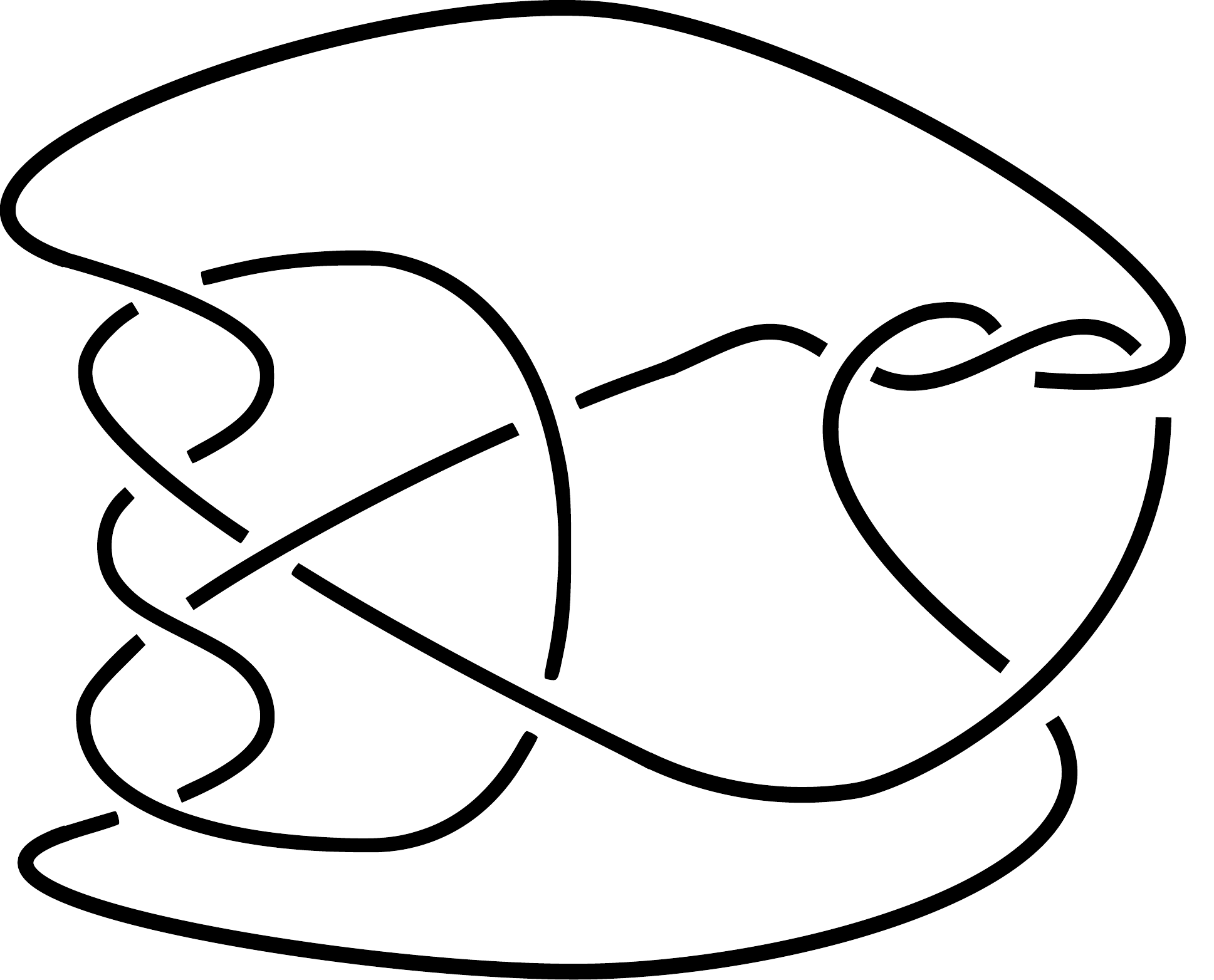}\qquad\includegraphics[scale=.08]{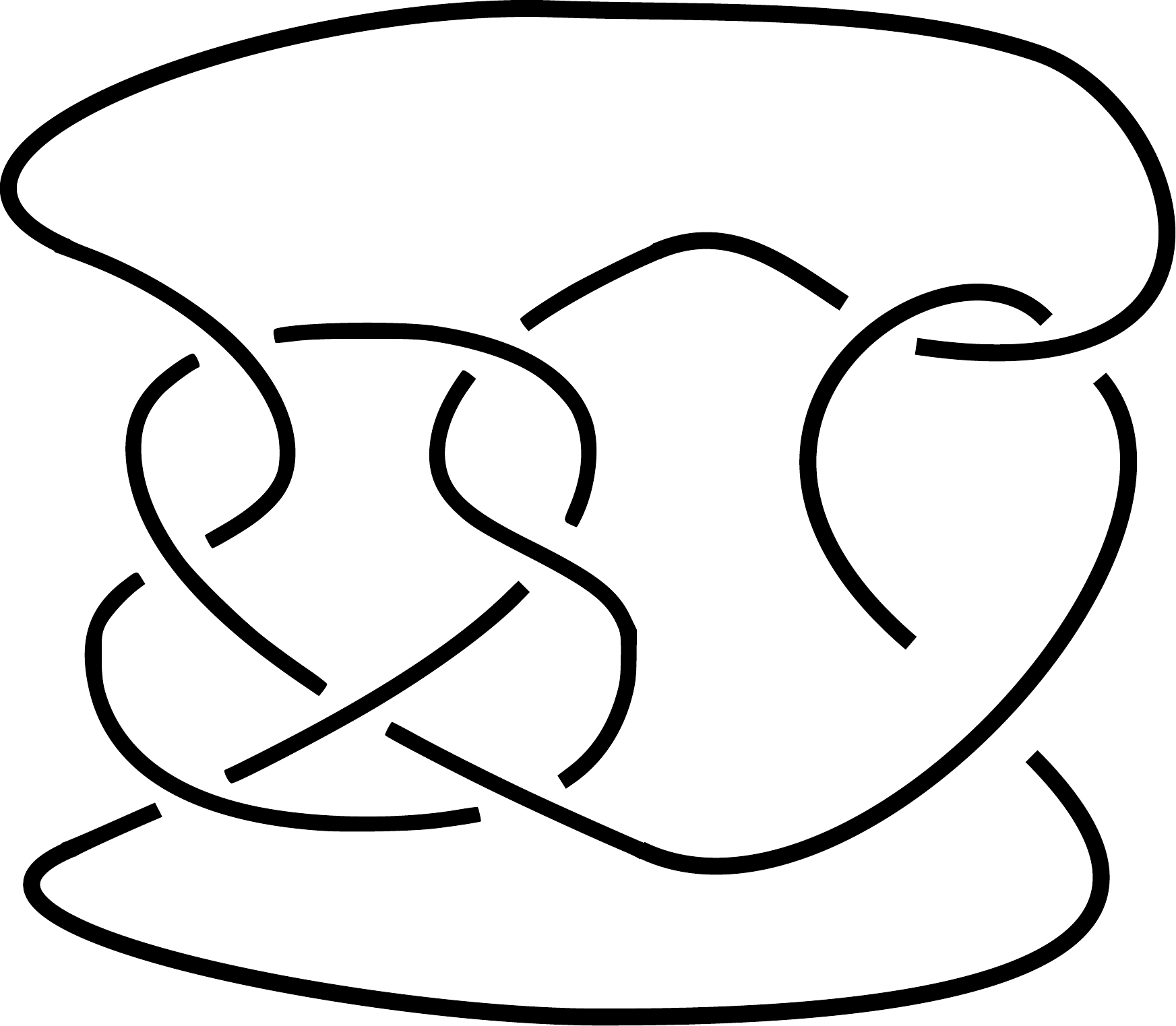}\qquad\includegraphics[scale=.08]{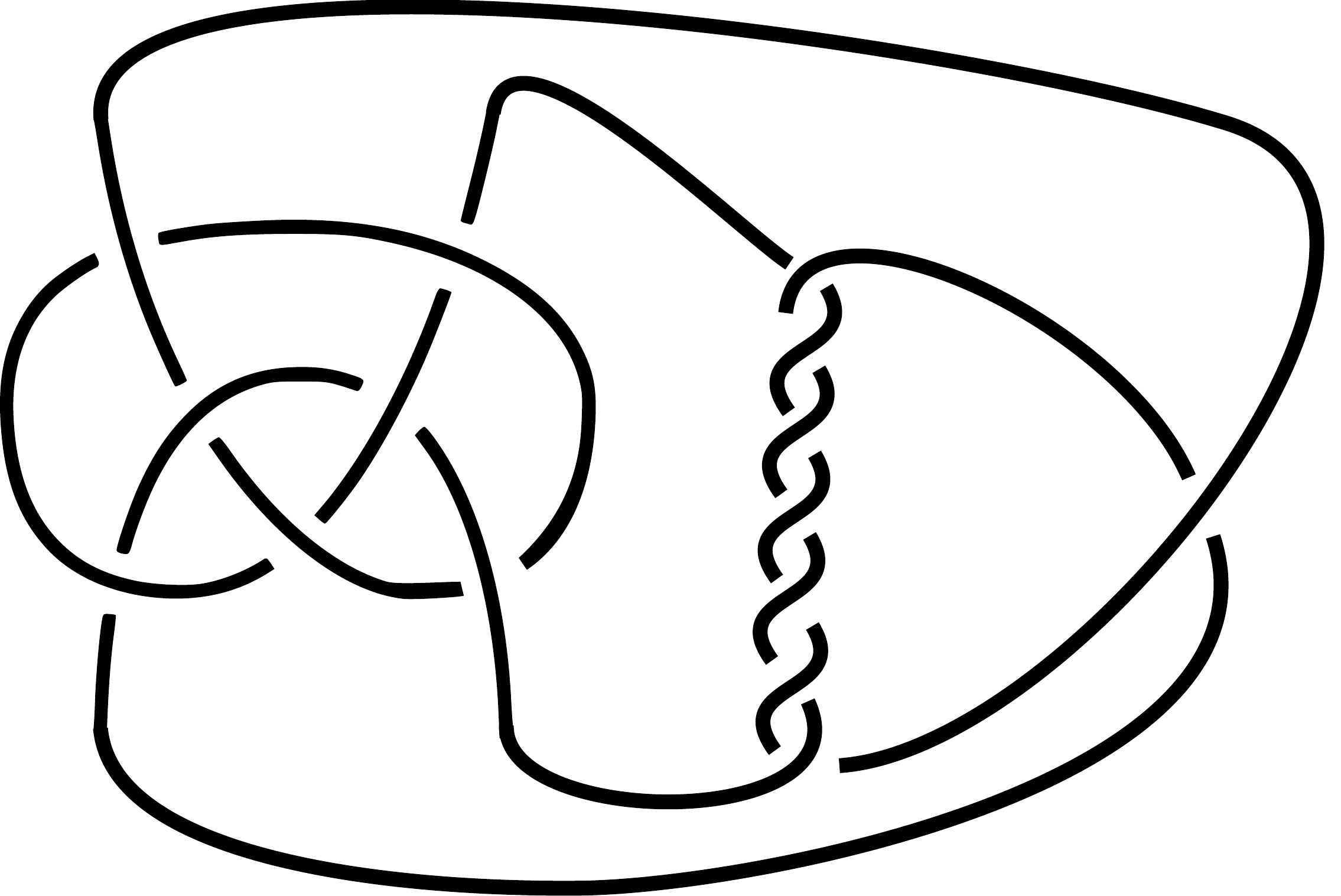}\qquad\includegraphics[scale=.08]{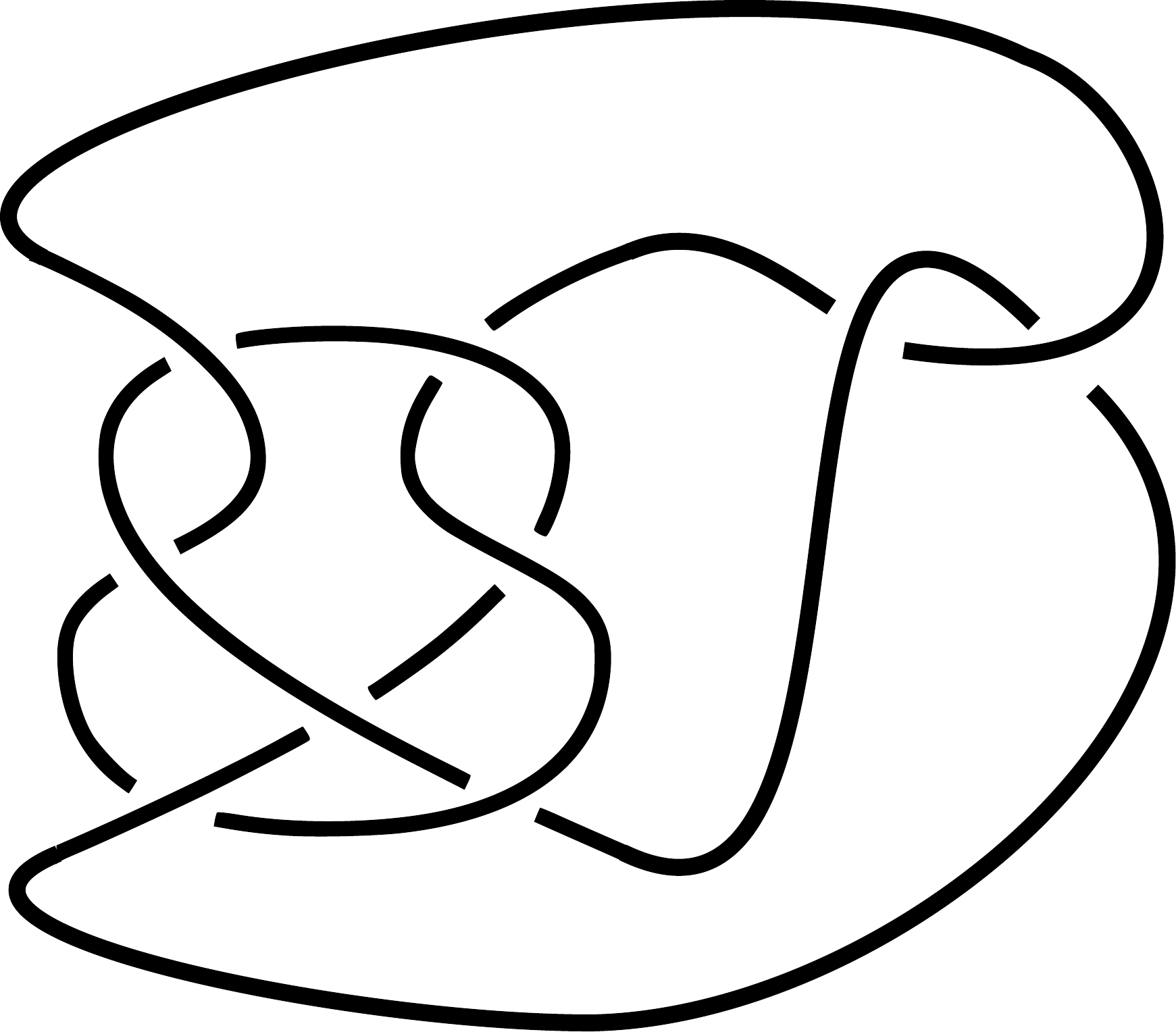}
		\caption{The links $N(7_{13}+\tangle{-3}{4})$, $N(7_{15}+\tangle{-2}{3})$, $N(7_{17}+\tangle{-8}{7})$, $N(7_{18}+[-2])$.}
		\label{figure:not_splittable}
	\end{figure}
\end{proof}

\vspace{.4cm}
\noindent
\address{\textsc{Department of Mathematics,\\
University of Coimbra}}\\
\email{\textit{E-mail:}\texttt{
		nogueira@mat.uc.pt}}

\vspace{.4cm}
\noindent
\address{\textsc{Department of Mathematics,\\
		University of Coimbra}}\\
\email{\textit{E-mail:}\texttt{
		ams@mat.uc.pt}}

\end{document}